\documentclass[letterpaper]{article}
\usepackage{graphicx}
\usepackage{amssymb,amsthm}
\usepackage{mathtools}
\usepackage{algorithm2e}
\usepackage{multirow}
\usepackage[section]{placeins}

\newcommand{\dx}{\,\mathrm{d}}
\newcommand{\R}{\mathbb{R}}
\newcommand{\N}{\mathbb{N}}
\newcommand{\functionSpace}{\mathcal{V}}
\newcommand{\Tmax}{T_\mathrm{max}}
\newcommand{\groundTruth}{x_g}
\newcommand{\Id}{\mathrm{Id}}

\newtheorem{theorem}{Theorem}[section]
\newtheorem{corollary}{Corollary}[theorem]
\newtheorem{remark}{Remark}

\begin{document}

\title{An Optimal Control Approach to Early Stopping Variational Methods for Image Restoration}
\author{Alexander Effland \and Erich Kobler \and Karl Kunisch \and Thomas Pock}

\maketitle

\begin{abstract}
We investigate a well-known phenomenon of variational approaches in image processing, where typically the best image
quality is achieved when the gradient flow process is stopped before converging to a stationary point.
This paradox originates from a tradeoff between optimization and modelling errors of the underlying variational model
and holds true even if deep learning methods are used to learn highly expressive regularizers from data.
In this paper, we take advantage of this paradox and introduce an optimal stopping time into the gradient flow process, which in turn is learned from data by means of an optimal control approach.
As a result, we obtain highly efficient numerical schemes that achieve competitive results for image denoising and image deblurring.
A nonlinear spectral analysis of the gradient of the learned regularizer gives enlightening insights about the different regularization properties.
\end{abstract}

\section{Introduction}
Throughout the past years, numerous image restoration tasks in computer vision
such as denoising~\cite{RuOsFa92}, segmentation~\cite{MuSh89} or super-resolution~\cite{ScLe15} have benefited from a variety of pioneering and novel variational methods.
In general, variational methods~\cite{ChPo16} are aiming at the minimization of an energy functional designed for a specific image reconstruction problem, where
the energy minimizer defines the restored output image.
In this paper, the energy functional is composed of an a priori known, task-dependent and quadratic data fidelity term and a Field of Experts type regularizer~\cite{RoBl09}, whose
building blocks are learned kernels and learned activation functions.
This regularizer generalizes the prominent total variation regularization functional and is capable of accounting for higher-order image statistics.
A classical approach to minimize the energy functional is a \emph{continuous-time gradient flow}, which defines a trajectory emanating from a fixed initial image.
Typically, the regularizer is adapted such that the end point image of the trajectory lies in a proximity of the ground truth image.
However, even the general class of Field of Experts type regularizers is not able to capture the entity of the complex structure of natural images,
that is why the end point image substantially differs from the ground truth image.
To address this insufficient modeling, we advocate an \emph{optimal control problem} using the gradient flow differential equation as the state equation and a cost functional that quantifies the distance of the
ground truth image and the gradient flow trajectory evaluated at the stopping time~$T$.
Besides the parameters of the regularizer, the stopping time is an additional control parameter learned from data.

The main contribution of this paper is the derivation of criteria to automatize the calculation of the \emph{optimal stopping time~$T$} for the aforementioned optimal control problem.
In particular, we observe that the learned stopping time is always finite even if the learning algorithm has the freedom to choose a larger stopping time.
This contradicts the variational paradigm that the quality of the outcome of iterative algorithms improves with larger values of~$T$.

For the numerical optimization, we discretize the state equation by means of the explicit Euler and Heun schemes.
This results in an iterative scheme which can be interpreted as static \emph{variational networks}~\cite{ChPo17,HaKl18,KoKl17} as a subclass of deep learning models~\cite{LeBe15}.
Here, the prefix "static" refers to constant regularizers with respect to time.
In several experiments we demonstrate the superiority of the learned static variational networks for image restoration tasks terminated at the optimal stopping time over classical variational methods.
Consequently, the early stopped gradient flow approach is better suited for image restoration problems and computationally more efficient than the classical variational approach.

A well-known major drawback of mainstream deep learning approaches is the lack of interpretability of the learned networks.
In contrast, following~\cite{Gi18}, the variational structure of the proposed model allows us to analyze the learned regularizers by means of a nonlinear spectral analysis.
The computed eigenpairs reveal insightful properties of the learned regularizers.

\medskip

There have been several approaches to cast deep learning models as dynamical systems in the literature, in which the model parameters can be seen as control parameters of an optimal control problem.
E~\cite{E17} clarified that deep neural networks such as residual networks~\cite{HeZh16} arise from a discretization of a suitable dynamical system.
In this context, the training process can be interpreted as the computation of the controls in the corresponding optimal control problem.
In \cite{LiHa18,LiCh18}, Pontryagin’s maximum principle is exploited to derive necessary optimality conditions for the optimal control problem in continuous time, which results in a rigorous discrete-time optimization.
Certain classes of deep learning networks are examined as mean-field optimal control problems in~\cite{EHa19}, where optimality conditions of the Hamilton--Jacobi--Bellman type and the Pontryagin type are derived.
The effect of several discretization schemes for classification tasks has been studied under the viewpoint of stability in \cite{HaRu17,ChMe18,BeCe19}, which leads to a variety of different network architectures that are empirically proven to be more stable.

\medskip

The benefit of early stopping for iterative algorithms is examined in the literature from several perspectives.
Raskutti et al.~\cite{RaWa11} exploit early stopping for non-para\-metric regression problems in reproducing kernel Hil\-bert spaces (RKHS) to prevent overfitting and derive a data-dependent stopping rule.
Yao et al.~\cite{YaRo07} discuss early stopping criteria for gradient descent algorithms for RKHS and relate these results to the Landweber iteration.
Quantitative properties of the early stopping condition for the Landweber iteration are presented in Binder et al.~\cite{BiHa96}.
Zhang and Yu \cite{ZhYu05} prove convergence and consistency results for early stopping in the context of boosting.
Prechelt~\cite{Pr12} introduces three heuristic criteria for optimal early stopping based on the performance of the training and validation error.
Rosasco and Villa~\cite{RoVi15} investigate early stopping in the context of incremental iterative regularization and prove sample bounds in a stochastic environment.
Matet et al.~\cite{MaRo17} exploit an early stopping method to regularize (strongly) convex functionals.
In contrast to these approaches, we propose early stopping on the basis of finding a local minimum with respect to the time horizon of a properly defined energy.

\medskip

To illustrate the necessity of early stopping for iterative algorithms, we revisit the established TV-$L^2$ denoising functional~\cite{RuOsFa92},
which amounts to minimizing the variational problem~$E[u]=\|u-g\|_{L^2(\Omega)}^2+\nu|Du|(\Omega)$
among all functions~$u\in BV(\Omega)$,
where $\Omega\subset\R^n$ denotes a bounded domain, $\nu>0$ is the regularization parameter and $g\in L^\infty(\Omega)$ refers to a corrupted input image.
An elementary, yet very inefficient optimization algorithm relies on a gradient descent using a finite difference discretization
for the regularized functional ($\epsilon>0$)
\begin{equation}
E_\epsilon[u_h]=\|u_h-g_h\|_{L^2(\Omega_h)}^2+\nu\sum_{(i,j)\in\Omega_h}\sqrt{|( D u_h)_{i,j}|^2+\epsilon^2},
\label{eq:TVL2}
\end{equation}
where $\Omega_h$ denotes a lattice, $u_h,g_h:\Omega_h\to\R$ are discrete functions and $( D u_h)_{i,j}$
is the finite difference gradient operator with Neumann boundary constraint (for details see \cite[Section~3]{ChCaNo09}).
For a comprehensive list of state-of-the-art methods to efficiently solve TV-based variational problems we refer the reader to \cite{ChPo16}.
\begin{figure}[htb]
\begin{center}
\includegraphics[width=.8\linewidth]{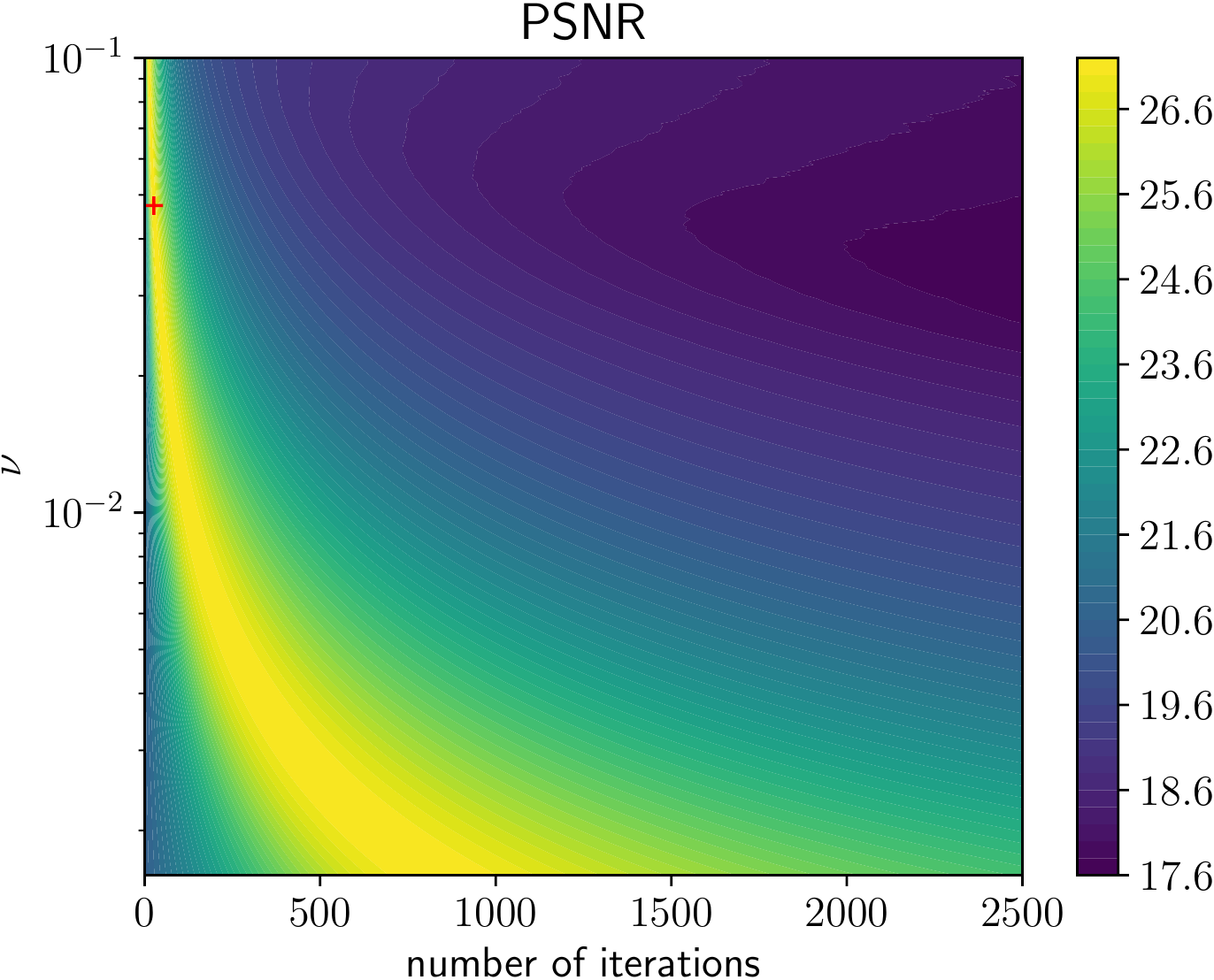}
\end{center}
\caption{Contour plot of the peak signal-to-noise ratio depending on the number of iterations and the regularization parameter~$\nu$ for TV-$L^2$ denoising.
The global maximum is marked with a red cross.}
\label{fig:TVL2}
\end{figure}
Figure~\ref{fig:TVL2} depicts the dependency of the peak signal-to-noise ratio on the number of iterations and the regularization parameter~$\nu$
for the TV-$L^2$ problem~\eqref{eq:TVL2} using a step size~$10^{-4}$ and $\epsilon=10^{-6}$,
where the input image~$g\in L^\infty(\Omega_h,[0,1])$ with a resolution of~$512\times 512$ is corrupted by additive Gaussian noise with standard deviation~$0.1$.
As a result, for each regularization parameter~$\nu$ there exists a unique optimal number of iterations, where the signal-to-noise ratio peaks.
Beyond this point, the quality of the resulting image is deteriorated by staircasing artifacts and fine texture patterns are smoothed out.
The global maximum~$(26,0.0474)$ is marked with a red cross, the associated image sequence is shown in Figure~\ref{fig:ROFsequence} (left to right: input image, noisy image, restored images after $13,26,39,52$~iterations).\footnote{
image by Nichollas Harrison (CC BY-SA 3.0)}
\begin{figure*}[htb]
\includegraphics[width=\linewidth]{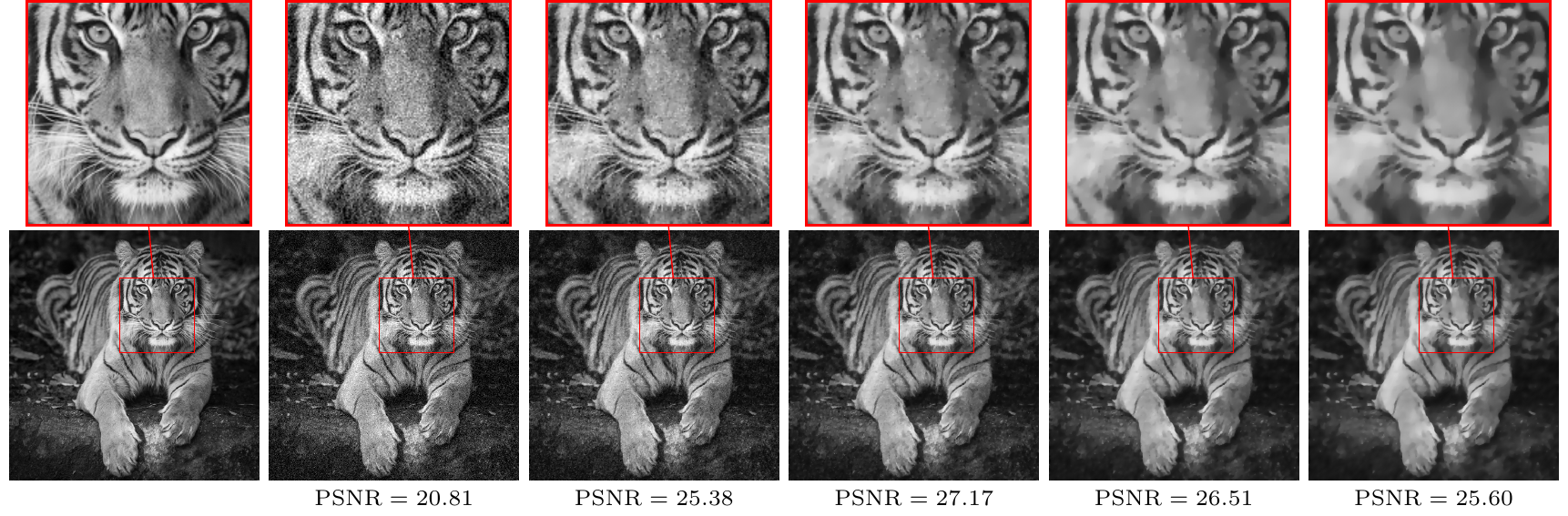}
\caption{Image sequence with globally best PSNR value. Left to right: input image, noisy image, restored images after $13,26,39,52$~iterations.}
\label{fig:ROFsequence}
\end{figure*}
If the gradient descent is considered as a discretization of a time continuous evolution process governed by a differential equation,
then the optimal number of iterations translates to an optimal stopping time.

\medskip

In this paper, we refer to the standard inner product in the Euclidean space~$\R^n$ by $\langle\cdot,\cdot\rangle$.
Let $\Omega\subset\R^n$ be a domain.
We denote the space of continuous functions by $C^0(\Omega)$, the space of $k$-times continuously differentiable functions by $C^k(\Omega)$ for $k\geq 1$, the Lebesgue space by $L^p(\Omega)$, $p\in[1,\infty)$, and the Sobolev space by $H^m(\Omega)=W^{m,2}(\Omega)$, $m\in\N$,
where the latter space is endowed with the Sobolev (semi-)norm for $f\in H^m(\Omega)$ defined as
$|f|_{H^m(\Omega)}=\|D^m f\|_{L^2(\Omega)}$ and $\|f\|_{H^m(\Omega)}=(\sum_{j=0}^m|f|_{H^j(\Omega)}^2)^\frac{1}{2}$.
With a slight abuse of notation we frequently set $C^k(\overline{\Omega})=C^0(\overline{\Omega})\cap C^k(\Omega)$.
The identity matrix in $\R^n$ is denoted by~$\Id$.
Finally, $\mathbf{1}=(1,\ldots,1)^\top\in\R^n$ is the one vector.

\medskip

This paper is organized as follows: 
In section~\ref{sec:timeCont}, we argue that certain classes of image restoration problems can be perceived as optimal control problems, in which
the state equation coincides with the evolution equation of static variational networks, and we prove the existence of solutions under quite general assumptions.
Moreover, we derive a first order necessary as well as a second order sufficient condition for the optimal stopping time in this optimal control problem.
A Runge--Kutta time discretization of the state equation results in the update scheme for static variational network, which is discussed in detail in section~\ref{sec:timeDisc}.
In addition, we visualize the effect of the optimality conditions in a simple numerical example in~$\R^2$ and discuss alternative approaches for the derivation of static variational networks.
Finally, we demonstrate the applicability of the optimality conditions to two prototype image restoration problems in section~\ref{sec:results}: denoising and deblurring.

\section{Optimal Control Approach to Early Stopping}\label{sec:timeCont}
In this section, we derive a time continuous analog of static variational networks as gradient flows of an energy functional~$\mathcal{E}$ composed of a data fidelity term~$\mathcal{D}$ and a Field of Experts type regularizer~$\mathcal{R}$.
The resulting ordinary differential equation is used as the state equation of an optimal control problem, in which the cost functional
incorporates the squared $L^2$-distance of the state evaluated at the optimal stopping time to the ground truth as well as (box) constraints of the norms of the stopping time, the kernels and the activation functions.
We prove the existence of minimizers of this optimal control problem under quite general assumptions.
Finally, we derive first and second order optimality conditions for the optimal stopping time using a Lagrangian approach.

Let $u\in\R^n$ be a data vector, which is either a signal of length~$n$ in 1D, an image of size $n=n_1\times n_2$ in 2D or spatial data of size $n=n_1\times n_2\times n_3$ in 3D.
Since we are primarily interested in two-dimensional image restoration, we focus on this task in the rest of this paper and merely remark that all results can be generalized to the remaining cases.
For convenience, we restrict to gray-scale images, the generalization to color or multi-channel images is straightforward.
In what follows, we analyze an \emph{energy functional} of the form
\begin{equation}
\mathcal{E}[u]=\mathcal{D}[u]+\mathcal{R}[u]
\label{eq:energyFunc}
\end{equation}
that is composed of a \emph{data fidelity term~$\mathcal{D}$} and an \emph{regularizer~$\mathcal{R}$} specified below.
We incorporate the \emph{Field of Experts regularizer}~\cite{RoBl09}, which is a common generalization of the discrete total variation regularizer and is given by
\[
\mathcal{R}[u]=\sum_{k=1}^{N_K}\sum_{i=1}^m\rho_k((K_ku)_i)
\]
with \emph{kernels~$K_k\in\R^{m\times n}$} and associated nonlinear functions~$\rho_k:\R\to\R$ for $k=1,\ldots,N_K$.

Throughout this paper, we consider the specific data fidelity term
\[
\mathcal{D}[u]=\frac{1}{2}\|Au-b\|_2^2
\]
for fixed $A\in\R^{l\times n}$ and fixed $b\in\R^l$.
We remark that various image restoration tasks can be cast in exactly this form for suitable choices of $A$ and $b$ \cite{ChPo16}.

The \emph{gradient flow}~\cite{AmGi08} associated with the energy~$\mathcal{E}$ for a time~$t\in(0,T)$ reads as
\begin{align}
\dot{\tilde{x}}(t)&=-D\mathcal{E}[\tilde{x}(t)]\notag\\
&=-A^\top(A\tilde{x}(t)-b)-\sum_{k=1}^{N_K}K_k^\top\Phi_k(K_k\tilde{x}(t)),\label{eq:gradientFlow}\\
\tilde{x}(0)&=x_0,\label{eq:gradientFlowInitial}
\end{align}
where $\tilde{x}\in C^1([0,T],\R^n)$ denotes the \emph{flow of $\mathcal{E}$} with $T\in\R$,
and the function $\Phi_k\in\functionSpace^s$ is given by
\[
(y_1,\ldots,y_m)^\top\mapsto(\rho_k'(y_1),\ldots,\rho_k'(y_m))^\top.
\]
For a fixed $s\geq 0$ and an a priori constant bounded open interval $I\subset\R$, we consider $C^s(\R,\R)$-conforming basis functions $\psi_1,\ldots,\psi_{N_w}$ with compact support in~$\overline{I}$ for $N_w\geq 1$.
The vectorial function space~$\functionSpace^s$ for the activation functions is composed of $m$ identical component functions~$\phi\in C^s(\R,\R)$, which
are given as the linear combination of $(\psi_j)_{j=1}^{N_w}$ with weight vector~$w\in\R^{N_w}$, i.e.
\begin{equation}    
\functionSpace^s\coloneqq\left\{\Phi=(\phi,\ldots,\phi):\R^m\to\R^m\Bigg|\phi=\sum_{j=1}^{N_w}w_j\psi_j\right\}.
\label{eq:functionSpace}
\end{equation}
We remark that in contrast to \cite{HaRu17,BeCe19}, inverse problems for image restoration rather than image classification are examined.
Thus, we incorporate in \eqref{eq:gradientFlow} the classical gradient flow with respect to the full energy functional in order to promote data consistency, whereas
in the classification tasks only the gradient flow with respect to the regularizer is considered.

In what follows, we analyze an optimal control problem, for which the state equation~\eqref{eq:gradientFlow} and initial condition~\eqref{eq:gradientFlowInitial} will arise as equality constraints.
The cost functional~$J$ incorporates the $L^2$-distance of the flow~$\tilde{x}$ evaluated at time~$T$ and the ground truth state~$\groundTruth\in\R^n$ and is given by
\[
\widetilde{J}(T,(K_k,\Phi_k)_{k=1}^{N_K})\coloneqq\frac{1}{2}\|\tilde{x}(T)-\groundTruth\|_2^2.
\]
We assume that the controls~$T$, $K_k$ and $\Phi_k$ satisfy the box constraints
\begin{equation}
0\leq T\leq\Tmax,\quad
\alpha(K_k)\leq 1,\quad
\beta(\Phi_k)\leq 1,
\label{eq:boxConstraints}        
\end{equation}
as well as the zero mean condition
\begin{equation}
K_k \mathbf{1}=0\in\R^m.
\label{eq:kernelConstraints}
\end{equation}
Here, we have $k=1,\ldots,N_K$ and we choose a fixed parameter $\Tmax>0$.
Further, $\alpha:\R^{m\times n}\to\R_0^+$ and $\beta:\functionSpace^s\to\R_0^+$ are continuously differentiable functions with non-vanishing gradient such that
$\alpha(K)\to\infty$ and $\beta(\Phi)\to\infty$ as $\|K\|\to\infty$ and $\|\Phi\|\to\infty$.
We include the condition~\eqref{eq:kernelConstraints} to reduce the dimensionality of the kernel space.
Moreover, this condition ensures an invariance with respect to gray-value shifts of image intensities.

The particular choice of the cost functional originates from the observation that a visually appealing image restoration is obtained as the closest point on the trajectory
of the flow~$\tilde{x}$ (reflected by the $L^2$-distance) to $\groundTruth$ subjected to a moderate flow regularization as quantified by the box constraints.
Figure~\ref{eq:trajectory} illustrates this optimization task for the optimal control problem.
Among all trajectories of the ordinary differential equation~\eqref{eq:stateEq} emanating from a constant initial value~$x_0$, one seeks the trajectory that is closest 
to the ground truth~$\groundTruth$ in terms of the squared Euclidean distance as visualized by the energy isolines.
Note that each trajectory is uniquely determined by $(K_k,\Phi_k)_{k=1}^{N_K}$.

The constraint of the stopping time is solely required for the existence theory.
For the image restoration problems, a finite stopping time can always be observed without constraints.
Hence, the optimal control problem reads as
\begin{equation}
\min_{T\in\R,K_k\in\R^{m\times n},\Phi_k\in\functionSpace^s}\widetilde{J}(T,(K_k,\Phi_k)_{k=1}^{N_K})
\label{eq:originalOptimalControl}
\end{equation}
subject to the constraints~\eqref{eq:boxConstraints} and \eqref{eq:kernelConstraints} as well as the nonlinear autonomous initial value problem (Cauchy problem) representing the state equation
\begin{align}
\dot{\tilde{x}}(t)=&f(\tilde{x}(t),(K_k,\Phi_k)_{k=1}^{N_K})\notag\\
\coloneqq&-A^\top(A\tilde{x}(t)-b)-\sum_{k=1}^{N_K}K_k^\top\Phi_k(K_k\tilde{x}(t))
\label{eq:stateOrig}
\end{align}
for $t\in(0,T)$ and $\tilde{x}(0)=x_0$.
We refer to the minimizing time~$T$ in \eqref{eq:originalOptimalControl} as the \emph{optimal early stopping time}.
To better handle this optimal control problem, we employ the reparametrization $x(t)=\tilde{x}(tT)$, which results in the equivalent optimal control problem
\begin{equation}
\min_{T\in\R,K_k\in\R^{m\times n},\Phi_k\in\functionSpace^s}J(T,(K_k,\Phi_k)_{k=1}^{N_K})
\label{eq:objectiveEq}
\end{equation}
subject to \eqref{eq:boxConstraints}, \eqref{eq:kernelConstraints} and the transformed state equation
\begin{equation}
\dot{x}(t)=Tf(x(t),(K_k,\Phi_k)_{k=1}^{N_K}),\qquad x(0)=x_0
\label{eq:stateEq}
\end{equation}
for $t\in(0,1)$, where
\[
J(T,(K_k,\Phi_k)_{k=1}^{N_K})\coloneqq\frac{1}{2}\|x(1)-\groundTruth\|_2^2.
\]
\begin{figure}
\begin{center}
\includegraphics[width=.8\linewidth]{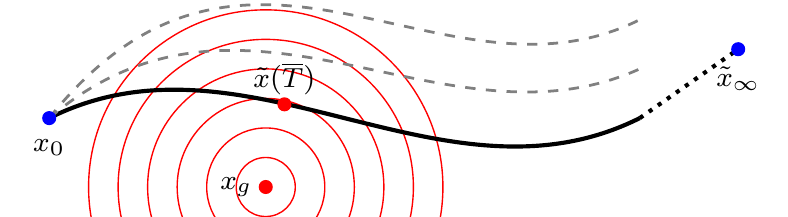}
\end{center}
\caption{Schematic drawing of optimal trajectory (black curve) as well as suboptimal trajectories (gray dashed curves) emanating from $x_0$ with ground truth~$\groundTruth$, optimal restored image~$\tilde{x}(\overline{T})$,
sink/stable node~$\tilde{x}_\infty$ and energy isolines (red concentric circles).}
\label{eq:trajectory}
\end{figure}
\begin{remark}
In the case of convex potential functions~$\rho$ with unbounded support, the autonomous differential equation~\eqref{eq:gradientFlow} is asymptotically stable.
\end{remark}
In the next theorem, we apply the direct method in the calculus of variations to prove the existence of minimizers for the optimal control problem.
\begin{theorem}[Existence of solutions]\label{thm:existenceSolution}
Let $s\geq 0$.
Then the minimum in \eqref{eq:objectiveEq} is attained.
\end{theorem}
\begin{proof}
Without restriction, we solely consider the case $N_K=1$ and omit the subscript.

Let $(T^i,K^i,\Phi^i)\in\R\times\R^{m\times n}\times\functionSpace^s$ be a minimizing sequence for~$J$ with an associated state~$x^i\in C^1([0,1],\R^n)$ such that \eqref{eq:boxConstraints}, \eqref{eq:kernelConstraints} and \eqref{eq:stateEq} hold true (the existence of~$x^i$ is verified below).
The coercivity of $\alpha$ and $\beta$ implies $\|K^i\|\leq C_\alpha$ and $\|\Phi^i\|\leq C_\beta$ for fixed constants $C_\alpha,C_\beta>0$.
Due to the finite dimensionality of $\functionSpace$ and the boundedness of $\|\Phi^i\|\leq C_\beta$ 
we can deduce the existence of a subsequence (not relabeled) such that $\Phi^i\to\Phi\in\functionSpace$.
In addition, using the bounds $T^i\in[0,\Tmax]$ and $\|K^i\|\leq C_\alpha$ we can pass to further subsequences if necessary to deduce
$(T^i,K^i)\to(T,K)$ for suitable $(T,K)\in[0,\Tmax]\times\R^{m\times n}$ such that $\|K\|\leq C_\alpha$.
The state equation~\eqref{eq:stateEq} implies
\begin{align*}
\|\dot{x}^i(t)\|_2
\leq&T^i\|K^i\|\|\Phi^i\|+T^i\|A\|_F(\|A\|_F\|x^i(t)\|_2+\|b\|_2)\\
\leq&\Tmax C_\alpha C_\beta+\Tmax\|A\|_F(\|A\|_F\|x^i(t)\|_2+\|b\|_2).
\end{align*}
This estimate already guarantees that $[0,1]$ is contained in the maximum domain of existence of the state equation due to the linear growth of the right-hand side in $x^i$ \cite[Theorem~2.17]{Te12}.
Moreover, Gronwall's inequality~\cite{Ha80,Te12} ensures the uniform boundedness of $\|x^i(t)\|_2$ for all $t\in[0,1]$ and all $i\in\N$, which in combination with
the above estimate already implies the uniform boundedness of $\|\dot{x}^i(t)\|_2$.
Thus, by passing to a subsequence (again not relabeled) we infer that $x\in H^1((0,1),\R^n)$ exists such that $x(0)=x_0$ (the pointwise evaluation is possible due to the Sobolev embedding theorem),
$x_i\rightharpoonup x$ in $H^1((0,1),\R^n)$ and $x_i\to x$ in $C^0([0,1],\R^n)$.
In addition, we obtain 
\[
\|T^i(K^i)^\top\Phi^i(K^i x^i(t))-TK^\top\Phi(Kx(t))\|_{C^0([0,1])}\to 0
\]
as $i\to\infty$ and $\dot{x}(t)=-TK^\top\Phi(Kx(t))-TA^\top(Ax(t)-b)$ holds true in a weak sense~\cite{Ha80}.
However, due to the continuity of the right-hand side we can even conclude $x\in C^1([0,1],\R^n)$ \cite[Chapter~I]{Ha80}.
Finally, the theorem follows from the continuity of $J$ along this minimizing sequence.
\end{proof}
In the next theorem, a first order necessary condition for the optimal stopping time is derived.
\begin{theorem}[First order necessary condition for optimal stopping time]\label{thm:firstOrderCondition}
Let $s\geq 1$.
Then for each stationary point~$(\overline{T},(\overline{K}_k,\overline{\Phi}_k)_{k=1}^{N_K})$ of~$J$
with associated state~$\overline{x}$ such that \eqref{eq:boxConstraints}, \eqref{eq:kernelConstraints} and \eqref{eq:stateEq} are valid the equation
\begin{equation}
\int_0^1\langle\overline{p}(t),\dot{\overline{x}}(t)\rangle\dx t=0\label{eq:optimalT}
\end{equation}
holds true.
Here, $\overline{p}\in C^1([0,1],\R^n)$ denotes the adjoint state of $\overline{x}$, which is given as the solution to the ordinary differential equation
\begin{equation}
\dot{\overline{p}}(t)=\sum_{k=1}^{N_K}\overline{T}\,\overline{K}_k^\top D\overline{\Phi}_k(\overline{K}_k\overline{x}(t))\overline{K}_k\overline{p}(t)+\overline{T}A^\top A\overline{p}(t)
\label{eq:adjointEq}
\end{equation}
with terminal condition
\begin{equation}
\overline{p}(1)=\groundTruth-\overline{x}(1).
\label{eq:terminalCondition}
\end{equation}
\end{theorem}
\begin{proof}
Again, without loss of generality we restrict to the case $N_K=1$ and omit the subscript.
Let $\overline{z}=(\overline{x},\overline{T},\overline{K},\overline{\Phi})\in\mathcal{Z}\coloneqq H^1((0,1))\times[0,\Tmax]\times\R^{m\times n}\times\functionSpace^s$ be a stationary point of $J$,
which exists due to Theorem~\ref{thm:existenceSolution}.
The constraints~\eqref{eq:stateEq}, \eqref{eq:boxConstraints} and \eqref{eq:kernelConstraints} can be written as 
\[
G(x,T,K,\Phi)\in\mathcal{C}\coloneqq\{0\}\times\{0\}\times\R_0^-\times\R_0^-\times\{0\},
\]
where $G:\mathcal{Z}\to\mathcal{P}\coloneqq L^2((0,1),\R^n)\times\R^n\times\R\times\R\times\R^m$ (note that $L^2((0,1),\R^n)^\top\cong L^2((0,1),\R^n)$) is given by
\[
G(x,T,K,\Phi)=
\begin{pmatrix}
\dot{x}+TK^\top\Phi(Kx)+TA^\top(Ax-b)\\
x(0)-x_0\\
\alpha(K)-1\\
\beta(\Phi)-1\\
K\mathbf{1}
\end{pmatrix}.
\]
For multipliers in the space $\mathcal{P}$ we consider the associated Lagrange functional~$L:\mathcal{Z}\times\mathcal{P}\to\R$
to minimize~$J$ incorporating the aforementioned constraints, i.e.~for $z=(x,T,K,\Phi)\in\mathcal{Z}$ and $p\in\mathcal{P}$ we have
\begin{equation}
L(z,p)=J(T,K,\Phi)+\int_0^1\langle p_1,G_1(z)\rangle\dx t+\sum_{i=2}^5\langle p_i,G_i(z)\rangle.
\label{eq:Lagrangian}
\end{equation}
Following \cite{ItKu08,Ze85}, the Lagrange multiplier~$\overline{p}$ exists if $J$ is Fr\'echet differentiable
at $\overline{z}$, $G$ is continuously Fr\'echet differentiable at $\overline{z}$ and $\overline{z}$ is regular, i.e.
\begin{equation}
0\in\operatorname{int}\left\{DG(\overline{z})(\mathcal{Z}-\overline{z})+G(\overline{z})-\mathcal{C}\right\}.
\label{eq:regularityMultiplier}
\end{equation}
The (continuous) Fr\'echet differentiability of $J$ and $G$ at~$\overline{z}$ can be proven in a straightforward manner. 
To show \eqref{eq:regularityMultiplier}, we first prove the surjectivity of $DG_1(\overline{z})$.
For any $z=(x,T,K,\Phi)\in\mathcal{Z}$ we have
\begin{align*}
DG_1(\overline{z})(z)=&\dot{x}+\overline{T}\,\overline{K}^\top D\overline{\Phi}(\overline{K}\overline{x})\overline{K}x+\overline{T}A^\top Ax
+T\overline{K}^\top\overline{\Phi}(\overline{K}\overline{x})
+TA^\top(A\overline{x}-b)\\
&+\overline{T}K^\top\overline{\Phi}(\overline{K}\overline{x})+\overline{T}\,\overline{K}^\top D\overline{\Phi}(\overline{K}\overline{x})K\overline{x}
+\overline{T}\,\overline{K}^\top\Phi(\overline{K}\overline{x}).
\end{align*}
The surjectivity of $DG_1(\overline{z})$ with initial condition given by $\overline{x}(0)=x_0$ follows from the linear growth in $x$, which implies that the maximum domain of existence coincides with~$\R$.
This solution is in general only a solution in the sense of Carath\'eodory~\cite{Ha80,Te12}.
Since $\alpha$ and $\beta$ have non-vanishing derivatives, the validity of \eqref{eq:regularityMultiplier} and thus the existence of the Lagrange multiplier follows.

The first order optimality conditions with test functions $x\in H^1((0,1),\R^n)$, $K\in\R^{m\times n}$, $\Phi\in\functionSpace^s$ and $p\in\mathcal{P}$ read as
\begin{align}
D_{x}L(\overline{x},\overline{T},\overline{K},\overline{\Phi},\overline{p})(x)
=&\langle\overline{x}(1)-\groundTruth,x(1)\rangle+\langle\overline{p}_2,x(0)\rangle\label{eq:opX}\\
&+\int_0^1\langle\overline{p}_1,\dot{x}+\overline{T}\,\overline{K}^\top D\overline{\Phi}(\overline{K}\overline{x})\overline{K}x+\overline{T}A^\top Ax\rangle\dx t=0,\notag\\
\frac{\dx}{\dx T}L(\overline{x},\overline{T},\overline{K},\overline{\Phi},\overline{p})
=&\int_0^1\langle\overline{p}_1,\overline{K}^\top\overline{\Phi}(\overline{K}\overline{x})+A^\top(A\overline{x}-b)\rangle\dx t=0,\label{eq:opT}\\
D_{K}L(\overline{x},\overline{T},\overline{K},\overline{\Phi},\overline{p})(K)
=&\int_0^1\langle\overline{p}_1,\overline{T}\,K^\top\overline{\Phi}(\overline{K}\overline{x})+\overline{T}\,\overline{K}^\top D\overline{\Phi}(\overline{K}\overline{x})K\overline{x}\rangle\dx t\notag\\
&+\langle\overline{p}_3,D\alpha(\overline{K})(K)\rangle+\langle\overline{p}_5,K\mathbf{1}\rangle=0,\notag\\
D_{\Phi}L(\overline{x},\overline{T},\overline{K},\overline{\Phi},\overline{p})(\Phi)
=&\int_0^1\langle\overline{p}_1,\overline{T}\,\overline{K}^\top\Phi(\overline{K}\overline{x})\rangle\dx t+\langle\overline{p}_4,D\beta(\overline{\Phi})(\Phi)\rangle=0,\notag\\
D_{p}L(\overline{x},\overline{T},\overline{K},\overline{\Phi},\overline{p})(p)
=&\int_0^1\langle p_1,G_1(\overline{z})\rangle\dx t+\sum_{i=2}^5\langle p_i,G_i(\overline{z})\rangle=0.\label{eq:opP}
\end{align}
The fundamental lemma of calculus of variations yields in combination with \eqref{eq:opX} and \eqref{eq:opP} for $t\in(0,1)$
\begin{align}
\dot{\overline{x}}(t)=&-\overline{T}\,\overline{K}^\top\overline{\Phi}(\overline{K}\overline{x}(t))-\overline{T}A^\top(A\overline{x}(t)-b),\label{eq:stateEqI}\\
\overline{x}(0)=&x_0,\notag\\
\dot{\overline{p}}_1(t)=&\overline{T}\,\overline{K}^\top D\overline{\Phi}(\overline{K}\overline{x}(t))\overline{K}\overline{p}_1(t)+\overline{T}A^\top A\overline{p}_1(t),\label{eq:stateEqII}\\
\overline{p}_1(1)=&\groundTruth-\overline{x}(1)\notag
\end{align}
in a distributional sense.
Since the right-hand sides of \eqref{eq:stateEqI} and \eqref{eq:stateEqII} are continuous if $s\geq 2$, we can conclude $\overline{x},\overline{p}\in C^1([0,1],\R^n)$ \cite{Ha80,Te12} and hence \eqref{eq:stateEqII} holds in the classical sense.
Finally, \eqref{eq:opT} and \eqref{eq:opP} imply
\begin{equation}
\frac{\dx}{\dx T}L(\overline{x},\overline{T},\overline{K},\overline{\Phi},\overline{p})=-\frac{1}{\overline{T}}\int_0^1\langle\overline{p}_1,\dot{\overline{x}}\rangle\dx t=0,
\label{eq:FOC}
\end{equation}
which proves \eqref{eq:optimalT} if $\overline{T}>0$ (the case $\overline{T}=0$ is trivial). 
\end{proof}
The preceding theorem can easily be adapted for fixed kernels and
activation functions leading to a reduced optimization problem with respect to the stopping time only:
\begin{corollary}[First order necessary condition for subproblem]\label{cor:firstOrderOptimality}
Let $\overline{K}_k\in\R^{m\times n}$ and $\overline{\Phi}_k\in \functionSpace^s$ for $k=1,\ldots,N_K$ be fixed with $s\geq 1$ satisfying \eqref{eq:boxConstraints} and \eqref{eq:kernelConstraints}.
We denote by $\overline{p}$ the adjoint state~\eqref{eq:adjointEq}.
Then, for each stationary point~$\overline{T}$ of the subproblem
\begin{equation}
T\mapsto J(T,(\overline{K}_k,\overline{\Phi}_k)_{k=1}^{N_K}),
\label{eq:subproblem}
\end{equation}
in which the associated state~$\overline{x}$ satisfies \eqref{eq:stateEq}, the first order optimality condition~\eqref{eq:optimalT} holds true.
\end{corollary}
\begin{remark}
Under the assumptions of Corollary~\ref{cor:firstOrderOptimality}, a re\-scaling argument reveals the identities for $t\in(0,1)$
\begin{align*}
\frac{\dx}{\dx T}\overline{x}(t)&=tf(\overline{x}(t),(\overline {K}_k,\overline{\Phi}_k)_{k=1}^{N_K}),\\
\frac{\dx}{\dx T}\overline{p}(t)&=\sum_{k=1}^{N_K}t\overline{K}_k^\top D\overline{\Phi}_k(\overline{K}_k\overline{x}(t))\overline{K}_k\overline{p}(t)+tA^\top A\overline{p}(t).    
\end{align*}
\end{remark}
We conclude this section with a second order sufficient condition for the partial optimization problem~\eqref{eq:subproblem}:
\begin{theorem}[Second order sufficient conditions for subproblem]\label{thm:secondOrderCondition}
Let $s\geq 2$.
Under the assumptions of Corollary~\ref{cor:firstOrderOptimality},
$\overline{T}\in(0,\Tmax)$ with associated state~$\overline{x}$ is a strict local minimum of~$T\mapsto J(T,(\overline{K}_k,\overline{\Phi}_k)_{k=1}^{N_K})$ if
a constant~$C>0$ exists such that
\begin{align}
&\int_0^1\sum_{k=1}^{N_K}\langle\overline{p},
\overline{T}\,\overline{K}_k^\top D^2\overline{\Phi}_k(\overline{K}_k\overline{x})(\overline{K}_kx,\overline{K}_kx)\notag\\
&+2\overline{K}_k^\top D\overline{\Phi}_k(\overline{K}_k\overline{x})\overline{K}_kx+2A^\top Ax\rangle\dx t+\langle x(1),x(1)\rangle\notag\\
\geq& C(1+\|x\|_{H^1((0,1),\R^n)}^2)
\label{eq:secondOrderFunctional}
\end{align}
for all $x\in C^1((0,1),\R^n)$ satisfying $x(0)=0$ and
\begin{equation}
\dot{x}=\sum_{k=1}^{N_K}\left(-\overline{T}\,\overline{K}_k^\top D\overline{\Phi}_k(\overline{K}_k\overline{x})\overline{K}_kx-\overline{K}_k^\top\overline{\Phi}_k(\overline{K}_k\overline{x})\right)
-\overline{T}A^\top Ax-A^\top(A\overline{x}-b).\label{eq:secondOrderConstraint}
\end{equation}
\end{theorem}
\begin{proof}
As before, we restrict to the case $N_K=1$ and omit subscripts.
Let us denote by~$L$ the version of the Lagrange functional~\eqref{eq:Lagrangian} with fixed kernels and fixed activation functions
to minimize $J$ subject to the side conditions~$G_1(x,T)=G_2(x,T)=0$ as specified in Corollary~\ref{cor:firstOrderOptimality}.
Let $\overline{z}=(\overline{x},\overline{T})\in\mathcal{Z}\coloneqq H^1((0,1),\R^n)\times(0,\Tmax)$ be a local minimum of~$J$.
Furthermore, we consider arbitrary test functions $z_1=(x_1,T_1), z_2=(x_2,T_2)\in\mathcal{Z}$, where we endow the Banach space~$\mathcal{Z}$
with the norm $\|z\|_\mathcal{Z}^2\coloneqq\|x\|_{H^1((0,1),\R^n)}^2+|T|^2$ for $z=(x,T)\in\mathcal{Z}$.
Then,
\begin{align*}
D^2J(\overline{T})(z_1,z_2)=&\langle x_1(1),x_2(1)\rangle,\\
D^2G_1(\overline{z})(z_1,z_2)=
&\overline{T}\,\overline{K}^\top D^2\overline{\Phi}(\overline{K}\overline{x})(\overline{K}x_1,\overline{K}x_2)
+T_2\,\overline{K}^\top D\overline{\Phi}(\overline{K}\overline{x})\overline{K}x_1\\
&+T_1\,\overline{K}^\top D\overline{\Phi}(\overline{K}\overline{x})\overline{K}x_2
+T_2A^\top Ax_1+T_1A^\top Ax_2,\\
D^2G_2(\overline{x})(x_1,x_2)=&0.
\end{align*}
Following \cite[Theorem~43.D]{Ze85}, $J$ has a strict local minimum at $\overline{z}$
if the first order optimality conditions discussed in Corollary~\ref{cor:firstOrderOptimality} holds true and
a constant~$C>0$ exists such that
\begin{align}
&D^2J(\overline{T})(z,z)+\int_0^1\langle\overline{p}_1,D^2G_1(\overline{z})(z,z)\rangle\dx t\notag\\
=&
\langle x(1),x(1)\rangle+\int_0^1\langle\overline{p}_1,
\overline{T}\,\overline{K}^\top D^2\overline{\Phi}(\overline{K}\overline{x})(\overline{K}x,\overline{K}x)\notag\\
&+2T\,\overline{K}^\top D\overline{\Phi}(\overline{K}\overline{x})\overline{K}x+2TA^\top Ax
\rangle\dx t
\geq C\|z\|_\mathcal{Z}^2
\label{eq:secondOrderCondition}
\end{align}
for all $z=(x,T)\in\mathcal{Z}$ satisfying
\begin{align*}
DG_1(\overline{z})(z)=&\dot{x}+\overline{T}\,\overline{K}^\top D\overline{\Phi}(\overline{K}\overline{x})\overline{K}x+T\,\overline{K}^\top\overline{\Phi}(\overline{K}\overline{x})\\
&+\overline{T}A^\top Ax+T A^\top(A\overline{x}-b)=0
\end{align*}
and $DG_2(\overline{z})(z)=x(0)=0$.
The theorem follows from the homogeneity of order~$2$ in~$T$ in~\eqref{eq:secondOrderCondition}, which results in the modified condition~\eqref{eq:secondOrderConstraint}.
\end{proof}
\begin{remark}
All aforementioned statements remain valid when replacing the function space~$\functionSpace^s$ and the norm of the activation functions by suitable Sobolev spaces and Sobolev norms, respectively.
Moreover, all statements only require minor modifications if instead of the box constraints~\eqref{eq:boxConstraints} 
non-negative, coercive and differentiable functions of the norms of $T$, $K_k$ and $\Phi_k$ are added in the cost functional~$J$.
\end{remark}

\section{Time Discretization}\label{sec:timeDisc}
The optimal control problem with state equation originating from the gradient flow for the energy functional~$\mathcal{E}$ was analyzed in section~\ref{sec:timeCont}.
In this section, we prove that static variational networks can be derived from a time discretization of the state equation incorporating Euler's or Heun's method~\cite{At89,Bu08}.
To illustrate the concepts, we discuss the optimal control problem in~$\R^2$ using fixed kernels and activation functions in section~\ref{sub:toyProblem}.
Finally, a literature overview of alternative ways to derive variational networks as well as relations to other approaches are presented in section~\ref{sub:VNReview}.

Let $S\geq 2$ be a fixed \emph{depth}.
For a stopping time~$T\in\R$ we define the \emph{node points $t_s=\frac{s}{S}$} for $s=0,\ldots,S$.
Consequently, Euler's explicit method for the transformed state equation~$\eqref{eq:stateEq}$ 
with fixed kernels and fixed activation functions $\overline{\Theta}=((\overline {K}_k,\overline{\Phi}_k)_{k=1}^{N_K})$
reads as
\begin{equation}
x_{s+1}=x_s+\frac{T}{S}f(x_s,\overline{\Theta})
\label{eq:Euler}
\end{equation}
for $s=0,\ldots,S-1$ with $x_0=x(0)$.
The discretized ordinary differential equation~\eqref{eq:Euler} defines the evolution of the \emph{static variational network}.
We stress that this time discretization is closely related to residual neural networks with constant parameters in each layer.
Here, $x_s$ is an approximation of $x(t_s)$, the associated global error~$x(t_s)-x_s$ is bounded from above by
\[
\max_{s=0,\ldots,S}\|x(t_s)-x_s\|_2\leq\frac{CT}{S}
\]
with $C\coloneqq\frac{\left(e^{L_f}-1\right)\|f''\|_{C^0}}{2L_f}$, where $L_f$ denotes the Lipschitz constant of~$f$ \cite[Theorem~6.3]{At89}.
In general, this global error bound has a tendency to overestimate the actual global error.
Improved error bounds can either be derived by performing a more refined local error analysis, which solely results in a better constant~$C$,
or by using higher order Runge--Kutta methods.
One prominent example of an explicit Runge--Kutta scheme with a quadratic order of convergence is Heun's method~\cite{Bu08}, which is defined as
\begin{equation}
x_{s+1}=x_s+\frac{T}{2S}\left(f(x_s,\overline{\Theta})+f\left(x_s+\frac{T}{S}f(x_s,\overline{\Theta})\right)\right).    
\label{eq:Heun}
\end{equation}
We abbreviate the right-hand side of \eqref{eq:adjointEq} as follows:
\begin{align*}
g(x,p,(K_k,\Phi_k)_{k=1}^{N_K})
=\sum_{k=1}^{N_K}K_k^\top D\Phi_k(K_kx)K_kp+A^\top Ap.    
\end{align*}
The corresponding update schemes for the adjoint sta\-tes are given by
\begin{equation}
p_s=p_{s+1}-\frac{T}{S}g(x_{s+1},p_{s+1},\overline{\Theta})
\label{eq:EulerAdjoint}
\end{equation}
in the case of Euler's method and
\begin{equation}
p_s=p_{s+1}-\frac{T}{2S}\bigg(g(x_{s+1},p_{s+1},\overline{\Theta})
+g\bigg(x_s,p_{s+1}-\frac{T}{S}g(x_{s+1},p_{s+1},\overline{\Theta}),\overline{\Theta}\bigg)\bigg)
\label{eq:HeunAdjoint}
\end{equation}
in the case of Heun's method.
We remark that in general implicit Runge--Kutta schemes are not efficient due to the complex structure of the Field of Experts regularizer.

In all cases, we have to choose the step size~$\frac{T}{S}$ such that the explicit Euler scheme is stable~\cite{Bu08}, i.e.
\begin{equation}
\max_{i=1,\ldots,n}\left|1+\frac{T}{S}\lambda_i\right|\leq 1    
\label{eq:stabilityCondition}
\end{equation}
for all $s=0,\ldots,S$, where $\lambda_i$ denotes the $i^{th}$ eigenvalue of the Jacobian of either $f$ or $g$.
Note that this condition already implies the stability of Heun's method.
Thus, in the numerical experiments we need to ensure a constant ratio of the stopping time~$T$ and the depth~$S$ to satisfy~\eqref{eq:stabilityCondition}.

\subsection{Optimal control problem in $\R^2$}\label{sub:toyProblem}
In this subsection, we apply the first and second order criteria for the partial optimal control problem (see Corollary~\ref{cor:firstOrderOptimality})
to the simple, yet illuminative example in~$\R^2$ with a single kernel, i.e.~$l=m=n=2$ and $N_K=1$.
More general applications of the early stopping criterion to image restoration problems are discussed in section~\ref{sec:results}.
Below, we consider a regularized data fitting problem composed of a squared $L^2$-data term and a nonlinear regularizer incorporating a forward finite difference matrix operator with respect to the $x$-direction.
In detail, we choose $\overline{\phi}(x)=\frac{x}{\sqrt{x^2 + 1}}$ and
\begin{align*}
x_0&=\begin{pmatrix} 1   \\ 2  \end{pmatrix},&
\groundTruth&=\begin{pmatrix} \frac{3}{2}\\ \frac{1}{2} \end{pmatrix},&
b&=\begin{pmatrix} 1\\ \frac{1}{2} \end{pmatrix},&\\
A &=\begin{pmatrix} 1 & 0\\0 & 1 \end{pmatrix},&
\overline{K}&=\begin{pmatrix} 1 & -1 \\ 0 & 0 \end{pmatrix}.&
\end{align*}
To compute the solutions of the state~\eqref{eq:stateEq} and the adjoint~\eqref{eq:adjointEq} differential equation, we use Euler's explicit method with $100$~equidistant steps.
All integrals are approximated using a Gaussian quadrature of order~$21$.
Furthermore, we optimize the stopping time~$\overline{T}$ in the discrete set $\mathcal{T}=0.05\cdot\N \cap[\frac{1}{10},3]$.

Figure~\ref{fig:toyProblem} (left) depicts all trajectories for $T\in\mathcal{T}$ (black curves) of the state equation emanating from~$x_0$ with sink/stable node~$x_\infty$.
The end points of the optimal trajectory and the ground truth state are marked by red points.
Moreover, the gray line indicates the trajectory of $\overline{x}_T+\overline{p}_T$ (the subscript denotes the solutions calculated with stopping time~$T$) associated with the optimal stopping time.
The dependency of the energy (red curve)
\[
T\mapsto J(T,\overline{K},\overline{\Phi})
\]
and of the first order condition~\eqref{eq:optimalT} (blue curve)
\[
T\mapsto-\frac{1}{T}\int_0^1\langle p_T,\dot{x}_T\rangle\dx t
\]
on the stopping time~$T$ is visualized in the right plot in Figure~\ref{fig:toyProblem}.
Note that the black vertical line indicating the optimal stopping time~$\overline{T}$ given by~\eqref{eq:optimalT} crosses the energy plot at the minimum point.
The function value of the second order condition~\eqref{eq:secondOrderFunctional} in Theorem~\ref{thm:secondOrderCondition} is $0.071$, which confirms that $\overline{T}$
is indeed a strict local minimum of the energy.
\begin{figure}[htb]
\includegraphics[width=\linewidth]{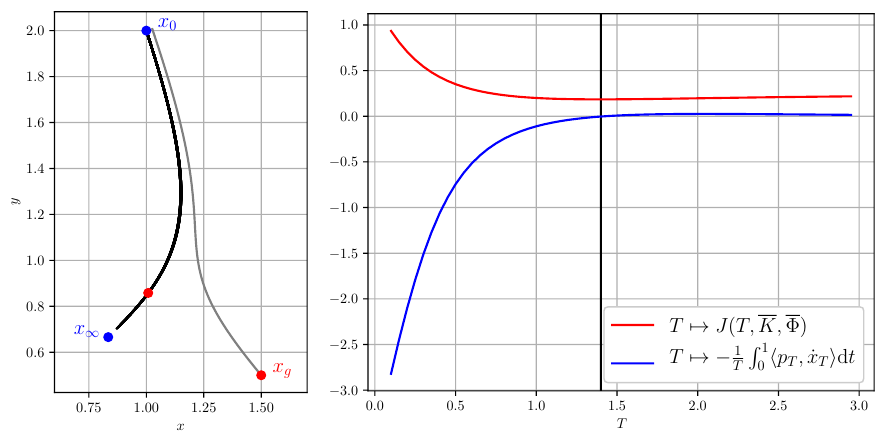}
\caption{Left: Trajectories of the state equation for~$T$ varying in~$\mathcal{T}$ (black curves), initial value~$x_0$ (blue point), sink/stable node~$x_\infty$ (blue point), ground truth state~$\groundTruth$ and
end point of optimal trajectory (red points).
 Right: function plots of the energy (red plot) and the first order condition (blue plot).}
\label{fig:toyProblem}
\end{figure}

\subsection{Alternative derivations of variational networks}\label{sub:VNReview}
We conclude this section with a brief review of alternative derivations of the defining equation~\eqref{eq:Euler} for variational networks.
Inspired by the classical nonlinear anisotropic diffusion model by Perona and Malik~\cite{PeMa90},
Chen and Pock~\cite{ChPo17} derive variational networks as discretized nonlinear reaction diffusion models of the form
$\frac{x_{s+1}-x_s}{h}=-\mathcal{R}[x_s]-\mathcal{D}[x_s]$ with an a priori fixed number of iterations, where $\mathcal{R}$ and $\mathcal{D}$ represent the reaction and diffusion terms, respectively,
that coincide with the first and second expression in \eqref{eq:stateEq}.
By exploiting proximal mappings this scheme can also be used for non-differentiable data terms~$\mathcal{D}$.
In the same spirit, Kobler et al.~\cite{KoKl17} related variational networks to incremental proximal and incremental gradient methods.
Following \cite{HaKl18}, variational networks result from a Landweber iteration~\cite{La51} of the energy functional~\eqref{eq:energyFunc} using the Field of Experts regularizer.
Structural similarities of variational networks and residual neural networks~\cite{HeZh16} are analyzed in \cite{KoKl17}.
In particular, residual neural networks (and thus also variational networks) are known to be less prone to the degradation problem,
which is characterized by a simultaneous increase of the training/test error and the model complexity.
Note that in most of these approaches time varying kernels and activation functions are examined.
In most of the aforementioned papers, the benefit of early stopping has been observed.

\section{Numerical Results for Image Restoration}\label{sec:results}
We examine the advantageousness of early stopping for image denoising and image deblurring using static variational networks in this section.
In particular, we show that the first order optimality condition results in the optimal stopping time.
We do not verify the second order sufficient condition discussed in Theorem~\ref{thm:secondOrderCondition} since in all experiments the first order condition
indicates an energy minimizing solution and thus this verification is not required.

\subsection{Image reconstruction problems}
In the case of \emph{image denoising}, we deteriorate a ground truth image~$\groundTruth\in\R^n$ by additive Gaussian noise
\[
n\sim\mathcal{N}(0,\sigma^2\Id)
\]
for a certain noise level~$\sigma$ resulting in the noisy input image~$g =\groundTruth+n$.
Consequently, the linear operator is given by the identity matrix and the corrupted image as well as the initial condition coincide with 
the noisy image, i.e.~$A=\Id$ and $b=x_0=g$.

For \emph{image deblurring}, we consider an input image $g=x_0=A\groundTruth+n\in\R^n$ that is corrupted by a Gaussian blur of the ground truth image~$\groundTruth\in\R^n$ and a Gaussian noise~$n$ with $\sigma=0.01$.
Here, $A\in\R^{n\times n}$ refers to the matrix representation of the $9\times 9$ normalized convolution filter with the blur strength~$\tau>0$ of the function
\[
(x,y)\mapsto\frac{1}{\sqrt{2\pi\tau^2}}\exp\left(-\frac{x^2+y^2}{2\tau^2}\right).
\]

\subsection{Numerical optimization}
For all image reconstruction tasks, we use the BSDS 500 data set~\cite{MaFo01} with gray-scale images in $[0,1]^{341\times421}$.
We train all models on 200 train and 200 test images from the BSDS 500 data set and evaluate the performance on 
68 validation images as specified by~\cite{RoBl09}.

In all experiments, the activation functions~\eqref{eq:functionSpace} are para\-me\-trized using $N_w=63$ quadratic B-spline basis functions~$\psi_j\in C^1(\R)$ 
with equidistant centers in the interval~$[-1,1]$.
Let $\xi\in\R^{n_1\times n_2}$ be the two-dimensional image of a corresponding data vector~$u\in\R^n$, $n=n_1\cdot n_2$.
Then, the convolution $\kappa\ast\xi$ of the image~$\xi$ with a filter~$\kappa$ is modeled by applying the corresponding kernel matrix~$K\in\R^{m\times n}$ to the data vector~$u$.
We only use kernels~$\kappa$ of size $7\times 7$.
Motivated by the relation $\|K\|_F=m\|\kappa\|_F$ we choose $\alpha(K)=\frac{1}{m^2}\|K\|_F^2$.
Additionally, we use $\beta(\Phi)=\beta(w)=\|w\|_2^2$ for a weight vector~$w$ associated with~$\Phi$.
Since all numerical experiments yield a finite optimal stopping time~$T$, we omit the constraint $T\leq\Tmax$.

For a given training set consisting of pairs of corrupted images~$x_0^i\in\R^n$ and corresponding ground truth images~$\groundTruth^i\in\R^n$, we denote the associated index set by~$\mathcal{I}$.
To train the model, we consider the discrete energy functional
\begin{align}
J_\mathcal{B}(T,(K_k,w_k)_{k=1}^{N_K})\coloneqq\frac{1}{|\mathcal{B}|}\sum_{i\in\mathcal{B}}\frac{1}{2}\|x_S^i-\groundTruth^i\|_2^2
\label{eq:training}
\end{align}
for a subset $\mathcal{B}\subset\mathcal{I}$, where $x_S^i$ denotes the terminal value of the Euler/Heun iteration scheme for the corrupted image~$x_0^i$.
In all numerical experiments, we use the iPALM algorithm~\cite{PoSa16} described in Algorithm~\ref{algo:optimization} to optimize all parameters with respect to a randomly selected batch~$\mathcal{B}$.
Each batch consists of $64$~image patches of size $96\times96$ that are uniformly drawn from the training data set.

For an optimization parameter~$q$ representing either $T$, $K_k$ or $w_k$, we use in the $l^{th}$~iteration step the over-relaxation
\[
\widetilde{q}^{[l]}=q^{[l]}+\frac{1}{\sqrt{2}}(q^{[l]}-q^{[l-1]}).
\]
We denote by $L_q$ the Lipschitz constant that is determined by backtracking and by $\operatorname{proj}_\mathcal{Q}$ the orthogonal projection onto the corresponding set denoted by~$\mathcal{Q}$.
\begin{algorithm}
\caption{iPALM algorithm for stochastic training of image restoration tasks for $L$~steps.}
\For{$l=1$ \KwTo $L$}{
\SetInd{1ex}{1ex}
randomly select batches $\mathcal{B}\subset\mathcal{I}$\;
update $x^i$ and $p^i$ for $i\in\mathcal{B}$ using either \eqref{eq:Euler}/\eqref{eq:EulerAdjoint} or \eqref{eq:Heun}/\eqref{eq:HeunAdjoint}\;
\For{$k=1$ \KwTo $N_K$}{
$K_k^{[l+1]}=\operatorname{proj}_{\mathcal{K}}\Big(\widetilde{K}_k^{[l]}-$ 
\hspace*{4em}$\frac{1}{L_K}  D_{K_k} J_{\mathcal{B}}(T^{[l]},(\widetilde{K}_k^{[l]},w_k^{[l]})_{k=1}^{N_K})\Big)$\;
}
\For{$k=1$ \KwTo $N_K$}{
$w_k^{[l+1]}=\operatorname{proj}_{\mathcal{W}}\Big(\widetilde{w}_k^{[l]}-$
\hspace*{4em}$\frac{1}{L_w}  D_{w_k} J_{\mathcal{B}}(T^{[l]},(K_k^{[l+1]},\widetilde{w}_k^{[l]})_{k=1}^{N_K})\Big)$\;
}
$T^{[l+1]}=\operatorname{proj}_{\R_0^+}\Big(\widetilde{T}^{[l]}-$
\hspace*{4em}$\frac{1}{L_T}  D_{T} J_{\mathcal{B}}(\widetilde{T}^{[l]},(K_k^{[l+1]},w_k^{[l+1]})_{k=1}^{N_K})\Big)$\;
}
\label{algo:optimization}
\end{algorithm}
Here, the constraint sets $\mathcal{K}$ and $\mathcal{W}$ are given by
\begin{align*}
\mathcal{K}&=\left\{K\in\R^{m\times n}:\alpha(K)\leq 1,K\mathbf{1}=0\right\},\\
\mathcal{W}&=\left\{w\in\R^{N_w}:\beta(w)\leq 1\right\}.
\end{align*}

Each component of the initial kernels~$K_k$ in the case of image denoising is independently drawn from a Gaussian random variable with mean~$0$ and variance~$1$ such that $K_k\in\mathcal{K}$.
The learned optimal kernels of the denoising task are incorporated for the initialization of the kernels for deblurring.
The weights~$w_k$ of the activation functions are initialized such that $\phi_k(y)\approx 0.1y$ around~$0$ for both reconstruction tasks.

\subsection{Results}
In the first numerical experiment, we train models for denoising and deblurring with $N_K=48$ kernels, a depth $S=10$ and $L=5000$ training steps.
Afterwards, we use the calculated parameters~$(K_k,\Phi_k)_{k=1}^{N_K}$ and $T$ as an initialization and train models for various depths $S=2,4,\ldots,50$ and $L=500$.
Figure~\ref{fig:denoisingEarlyStopping} depicts the average PSNR value~$\overline{\mathrm{PSNR}(x_S^i,x_g^i)}_{i\in\widehat{\mathcal{I}}}$ with $\widehat{\mathcal{I}}$ denoting the index set of the test images and the learned stopping time~$T$ as a function of the depth~$S$ for denoising (first two plots)
and deblurring (last two plots).
As a result, we observe that all plots converge for large~$S$, where the PSNR curve monotonically increases.
Moreover, the optimal stopping time~$T$ is finite in all these cases, which empirically validates that early stopping is beneficial.
Thus, we can conclude that beyond a certain depth~$S$ the performance increase in terms of PSNR is negligible and a proper choice of the optimal stopping time is significant.
The asymptotic value of~$T$ for increasing~$S$ in the case of image deblurring is approximately $20$~times larger compared to image denoising due to the structure of the deblurring operator~$A$.
Figure~\ref{fig:normsKernels} depicts the average $\ell^2$-difference of consecutive convolution kernels and activation functions for denoising and deblurring as a function of the depth~$S$.
We observe that the differences decrease with larger values of~$S$, which is consistent with the convergence of the optimal stopping time~$T$ for increasing~$S$.
Both time discretization schemes perform similar and thus in the following experiments we solely present results calculated with Euler's method due to advantages in the computation time.
\begin{figure}[htb]
\centering
\includegraphics[width=.5\linewidth]{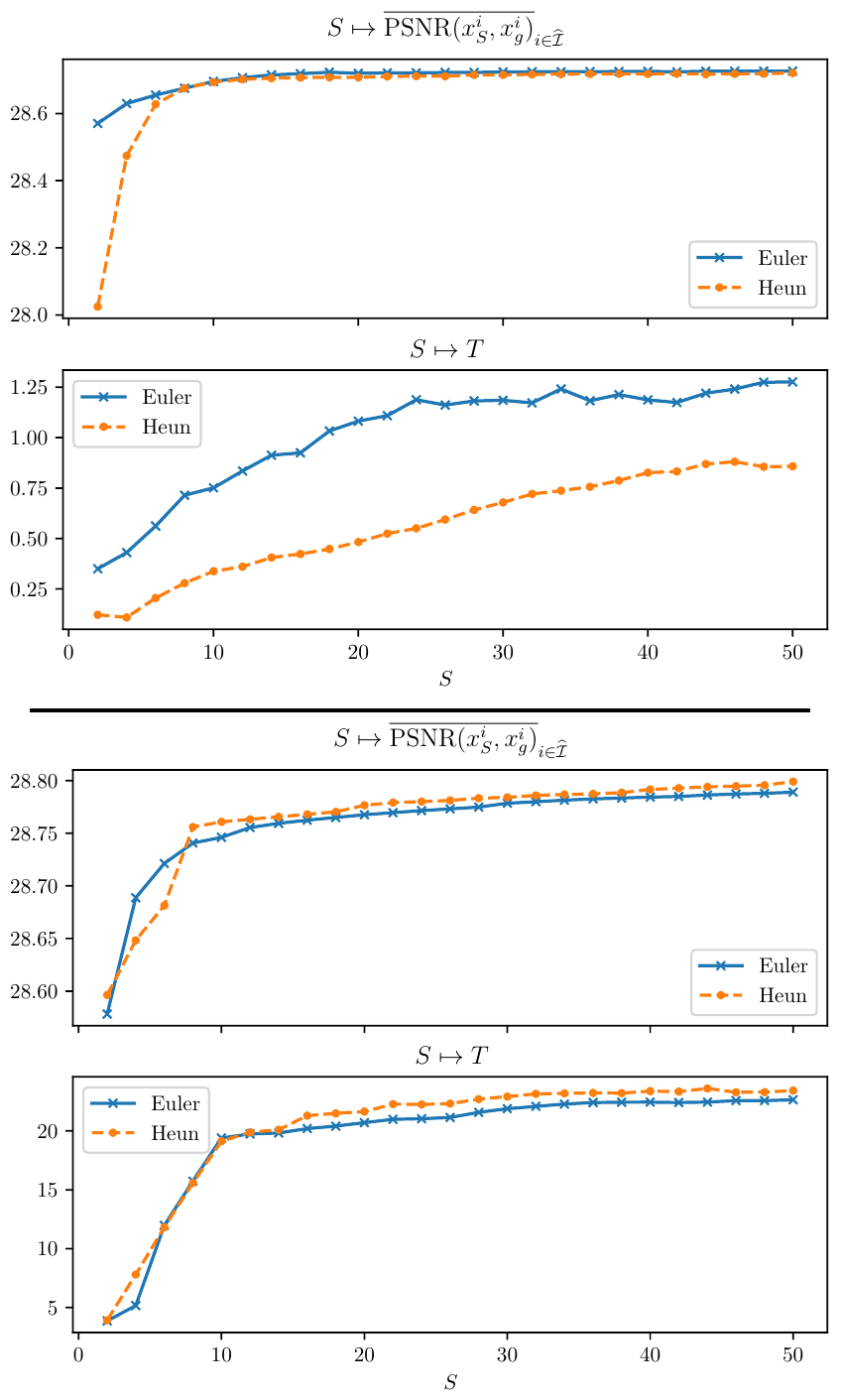}
\caption{Plots of the average PSNR value across the test set (first and third plot) as well as the learned optimal stopping time~$T$ (second and fourth plot) as a function of 
the depth~$S$ for denoising (top) and deblurring (bottom).
All plots show the results for the explicit Euler and explicit Heun schemes.}
\label{fig:denoisingEarlyStopping}
\end{figure}

\begin{figure}[htb]
\includegraphics[width=\linewidth]{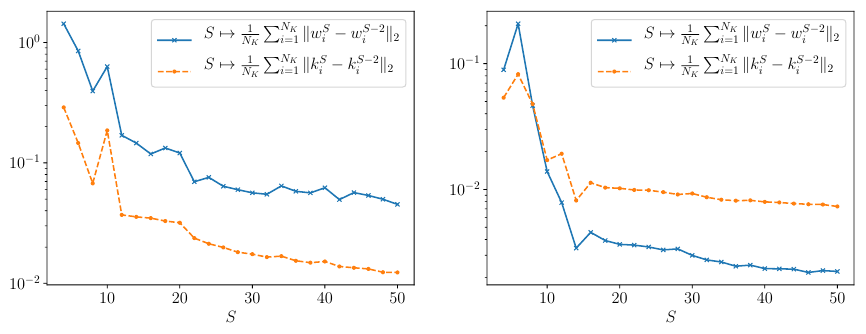}
\caption{Average change of consecutive convolution kernels (solid blue) and activation functions (dotted orange) for denoising (left) and deblurring (right) in terms of the $\ell^2$-norm.}
\label{fig:normsKernels}
\end{figure}

We next demonstrate the applicability of the first order condition for the energy minimization in static variational networks
using Euler's dis\-creti\-zation scheme with $S=20$.
Figure~\ref{fig:firstOrderCondition} depicts band plots along with the average curves among all training/test images of the functions
\begin{align}
T&\mapsto J_{\{i\}}(T,(\overline{K}_k, \overline{w}_k)_{k=1}^{N_K})\text{ and}\label{eq:energyCurvePartial}\\
T&\mapsto -\frac{1}{T}\int_0^1 \langle p_T^i, \dot{x}_T^i \rangle \dx t\notag
\end{align}
for all training and test images for denoising (first two plots) and deblurring (last two plots).
We approximate the integral in the first order condition~\eqref{eq:optimalT} via
\[
\int_0^1\langle p(t),\dot{x}(t)\rangle\dx t\approx\frac{1}{S+1}\sum_{s=0}^S\langle p_s,Tf(x_s,(K_k,\Phi_k)_{k=1}^{N_K})\rangle.
\]
We deduce that the first order condition for each test image indicates the energy minimizing stopping time.
Note that all image dependent stopping times are distributed around the average optimal stopping time that is highlighted by the black vertical line
and learned during training.
\begin{figure}[htb]
\centering
\includegraphics[width=.5\linewidth]{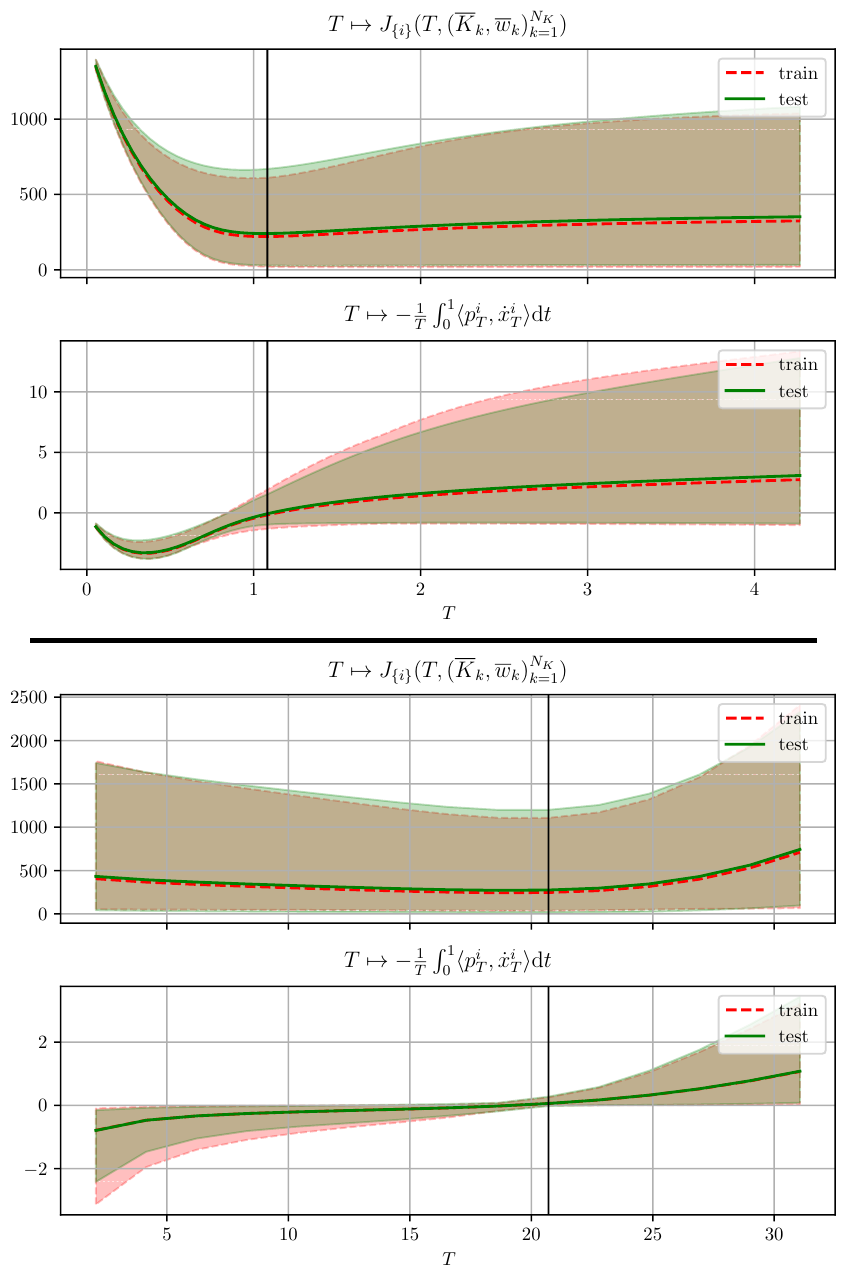}
\caption{Plots of the energies (first and third plot) and first order conditions (second and fourth plot) for training and test set with $\sigma=0.1$ for denoising (first pair of plots) and $\tau=1.5$/$\sigma=0.01$ for deblurring (last pair of plots).
The average value across the training/test sets are indicated by the dotted red/solid green curves.
The area between the minimal and maximal function value for each $T$ across the training/test set are indicated by the red dotted area and the green cross-hatched area, respectively.}
\label{fig:firstOrderCondition}
\end{figure}
Figure~\ref{fig:EulerResults} depicts two input images~$\groundTruth$ (first column), the corrupted images~$g$ (second column) as well as the denoised images for 
$T=\frac{\overline{T}}{2},\overline{T},\frac{\overline{3T}}{2},100\overline{T}$ (third to sixth column).
The maximum values of the PSNR values obtained for $T=\overline{T}$ are $29.68$ and $29.52$, respectively.
To ensure a sufficiently fine time discretization, we enforce $\frac{S}{T}=\mathrm{const}$, where for $T=\overline{T}$ we set $S=20$.
Likewise, Figure~\ref{fig:EulerResultsDeblur} contains the corresponding results for the deblurring task.
Again, we enforce a constant ratio of $S$ and $T$.
The PSNR value peaks around the optimal stopping time, the corresponding values are $29.52$ and $27.80$, respectively.
We observed an average computation time of $5.694~ms$ for the denoising and $8.687~ms$ for the deblurring task using a RTX 2080 Ti graphics card and the PyTorch machine learning framework.

As desired, $\overline{T}$ indicates the energy minimizing time, where both the average curves for the training and test sets nearly coincide, which proves that the model generalizes to unseen test images.
Although the gradient of the average energy curve~\eqref{eq:energyCurvePartial} is rather flat near the learned optimal stopping time, the proper choice of~$\overline{T}$
is indeed crucial as shown by the qualitative results in Figure~\ref{fig:EulerResults} and Figure~\ref{fig:EulerResultsDeblur}.
In the case of denoising, for $T<\overline{T}$ we still observe noisy images, whereas for too large~$T$ local image patterns are smoothed out.
For image deblurring, images computed with too small values of $T$ remain blurry, while for $T>\overline{T}$ ringing artifacts are generated and their intensity increase with larger~$T$.
For a corrupted image, the associated adjoint state requires the knowledge of the ground truth for the terminal condition~\eqref{eq:terminalCondition}, which is in general not available.
However, Figure~\ref{fig:firstOrderCondition} shows that the learned average optimal stopping time~$\overline{T}$ yields the smallest expected error.
Thus, for arbitrary corrupted images $\overline{T}$ is used as the stopping time.

\begin{figure*}[htb]
\includegraphics[width=\linewidth]{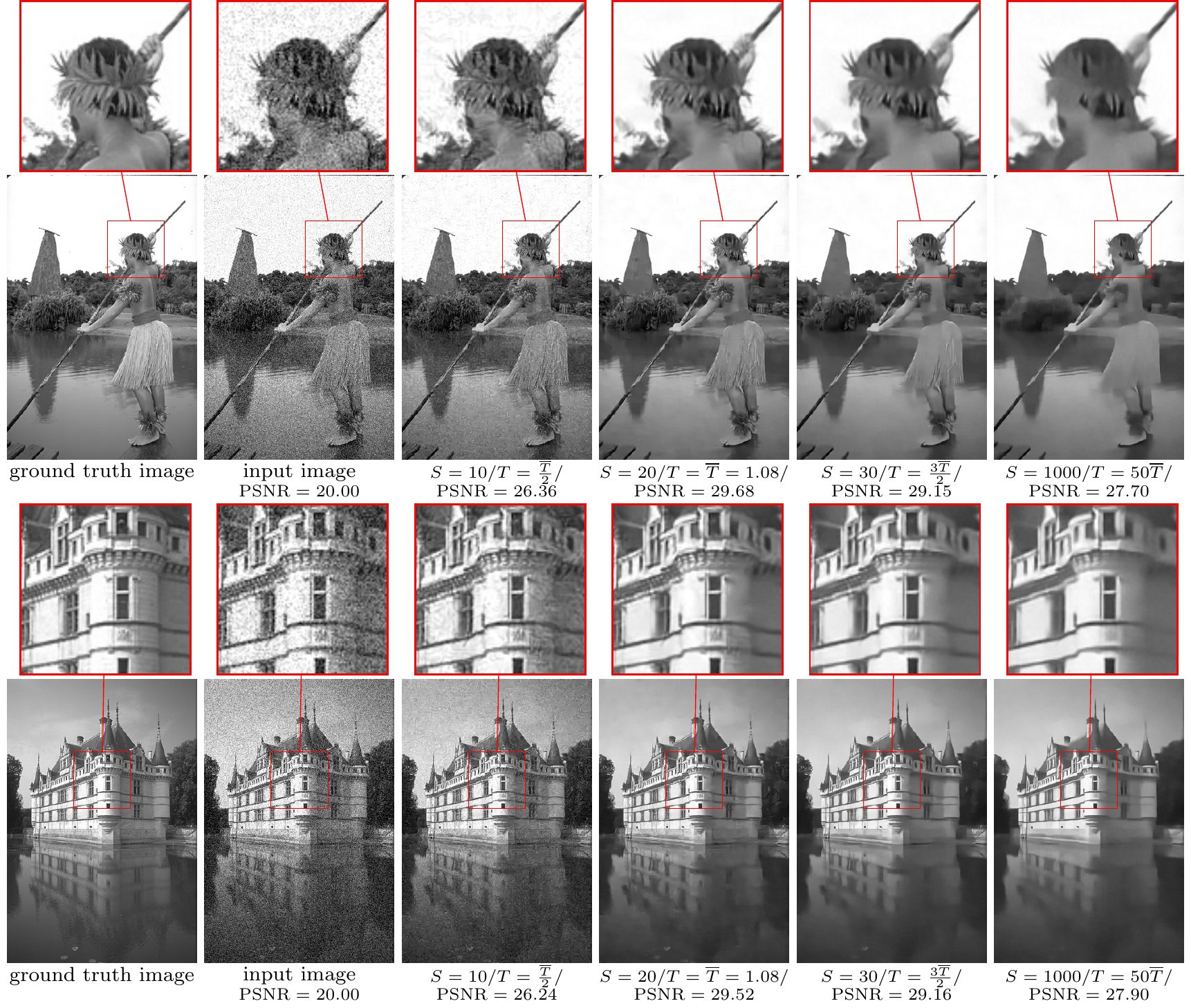}
\caption{From left to right: ground truth image, noisy input image ($\sigma=0.1$), restored images for $T=\frac{\overline{T}}{2},\overline{T},\frac{\overline{3T}}{2},50\overline{T}$ with $\overline{T}=1.08$ for image denoising.}
\label{fig:EulerResults}
\end{figure*}

\begin{figure*}[htb]
\includegraphics[width=\linewidth]{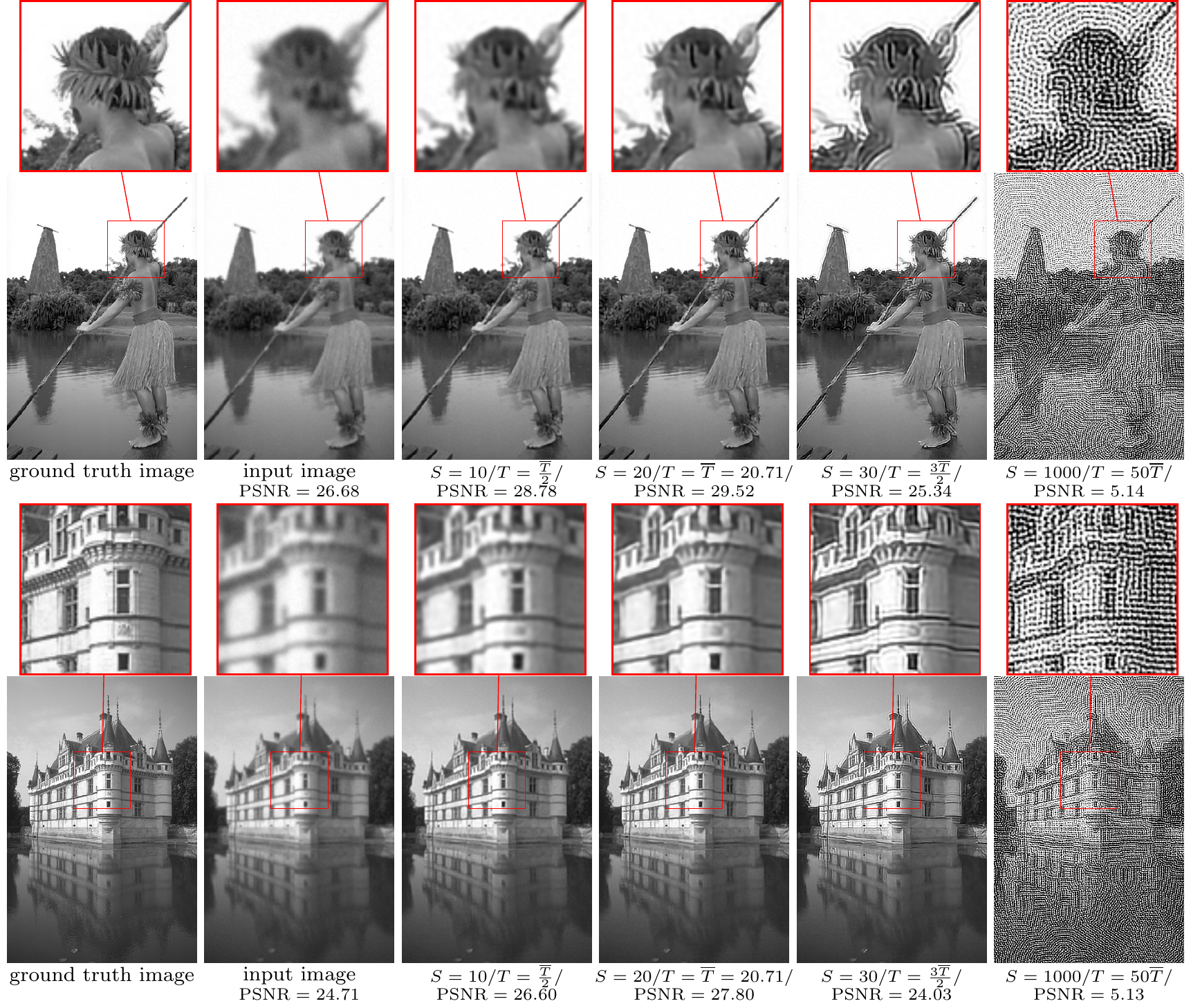}
\caption{From left to right: ground truth image, blurry input image ($\tau=1.5$, $\sigma=0.01$), restored images for $T=\frac{\overline{T}}{2},\overline{T},\frac{\overline{3T}}{2},50\overline{T}$ with $\overline{T}=20.71$ for image deblurring.}
\label{fig:EulerResultsDeblur}
\end{figure*}

Figure~\ref{fig:differentNoise} illustrates the plots of the energies (blue plots) and the first order conditions (red plots) as a function of the stopping time~$T$ for all test images for denoising (left) and deblurring (right), which are degraded by noise levels~$\sigma\in\{0.075,0.1,0.125,0.15\}$ and different blur strengths~$\tau\in\{1.25,1.5,1.75,2.0\}$.
Note that in each plot the associated curves of three prototypic images are visualized.
To ensure a proper balancing of the data fidelity term and the regularization energy for the denoising task, we add the factor $\frac{1}{\sigma^2}$ to the data term as typically motivated by Bayesian inference.
For all noise levels~$\sigma$ and blur strengths~$\tau$, the same fixed pairs of kernels and activation functions trai\-ned with $\sigma=0.1$/$\tau=1.5$ and depth~$S=20$ are used.
Again, the first order conditions indicate the degradation depending energy minimizing stopping times.
The optimal stopping time increases with the noise level and blur strength, which results from a larger distance of~$x_0$ and $\groundTruth$ and thus requires longer trajectories.
\begin{figure*}[htb]
\includegraphics[width=\linewidth]{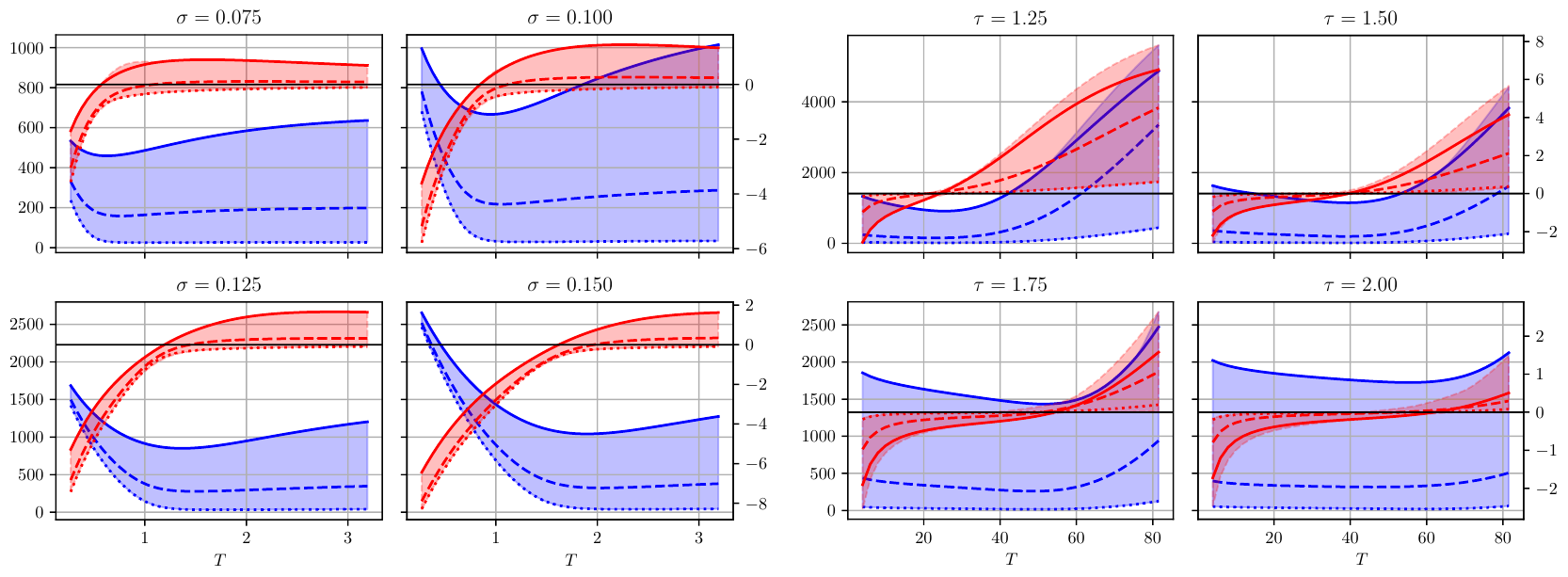}
\caption{Band plots of the energies (blue plots) and first order conditions (red plots) for image denoising (left) and image deblurring (right) and various degradation levels.
In each plot, the curves of three prototypic images are shown.}
\label{fig:differentNoise}
\end{figure*}

Table~\ref{tab:tablePSNR} presents pairs of average PSNR values and optimal stopping times~$\overline{T}$ for the test set for denoising (top) and deblurring (bottom) for different noise levels~$\sigma\in\{0.075,0.1,0.125,0.15\}$ and blur strengths~$\tau\in\{1.25,1.5,1.75,2.0\}$.
All results in the table are obtained using $N_K=48$ kernels and a depth~$S=20$.
Both first rows present the results incorporating an optimization of all control parameters~$(K_k, w_k)_{k=1}^{N_K}$ and $T$.
In contrast, both second rows show the resulting PSNR values and optimal stopping times for only a partial optimization of the stopping times and pretrained kernels and activation functions~$(K_k, w_k)_{k=1}^{N_K}$ for~$\sigma=0.1$/$\tau=1.5$.
Interestingly, the resulting PSNR values are almost identical for image denoising despite varying optimal stopping times.
Consequently, a model that was pretrained for a specific noise level can be easily adapted to noise levels by modifying the optimal stopping time.
However, in the case of image deblurring the model benefits from a full optimization of all controls, which is caused by the dependency of~$A$ on the blur strength.
For the noise level~$0.1$ we observe the average PSNR value~$28.72$ which is on par with the corresponding results of~\cite[Table~II]{ChRa15}.
We emphasize that in their work a costly full minimization of an energy functional is performed, whereas we solely require a depth~$S=20$ to compute comparable results.
\begin{table*}
\caption{Average PSNR value of the test set for image denoising/deblurring with different degradation levels along with the optimal stopping time~$\overline{T}$.
The first rows of each table present the results obtained by the optimization of all control parameters.
The second rows show the results calculated with fixed $(K_k, w_k)_{k=1}^{N_K}$, which were pretrained for~$\sigma=0.1$/$\tau=1.5$.
}
\label{tab:tablePSNR}
\centering
\resizebox{\linewidth}{!}{
\begin{tabular}{l|cc|cc|cc|cc}
\multicolumn{9}{c}{\textbf{image denosing}}\\[.5em] 
 & \multicolumn{2}{c|}{$\sigma=0.075$}& \multicolumn{2}{c|}{$\sigma=0.1$}& \multicolumn{2}{c|}{$\sigma=0.125$}& \multicolumn{2}{c}{$\sigma=0.15$} \\
 & $\overline{\mathrm{PSNR}}$ & $\overline{T}$ & $\overline{\mathrm{PSNR}}$ & $\overline{T}$ & $\overline{\mathrm{PSNR}}$ & $\overline{T}$ & $\overline{\mathrm{PSNR}}$ & $\overline{T}$ \\
\noalign{\smallskip}\hline\noalign{\smallskip}
full optimization of all controls & 30.05 & 0.724 & \multirow{2}{*}{28.72} & \multirow{2}{*}{1.082} & 27.72 & 1.445 & 26.95 & 1.433 \\
optimization only of~$\overline{T}$ & 30.00 & 0.757 &  &  & 27.73 & 1.514 & 26.95 & 2.055 \\
\noalign{\smallskip}\hline
\end{tabular}
}
\vspace{1em}
\resizebox{\linewidth}{!}{
\begin{tabular}{l|cc|cc|cc|cc}
\multicolumn{9}{c}{\textbf{image deblurring}}\\[.5em] 
& \multicolumn{2}{c|}{$\tau=1.25$}& \multicolumn{2}{c|}{$\tau=1.5$}& \multicolumn{2}{c|}{$\tau=1.75$}& \multicolumn{2}{c}{$\tau=2.0$} \\
& $\overline{\mathrm{PSNR}}$ & $\overline{T}$ & $\overline{\mathrm{PSNR}}$ & $\overline{T}$ & $\overline{\mathrm{PSNR}}$ & $\overline{T}$ & $\overline{\mathrm{PSNR}}$ & $\overline{T}$ \\
\noalign{\smallskip}\hline\noalign{\smallskip}
full optimization of all controls & 29.95 & 39.86 & \multirow{2}{*}{28.76} & \multirow{2}{*}{37.78} & 27.87 & 40.60 & 27.13 & 40.01 \\
optimization only of~$\overline{T}$  & 29.73 & 23.86 &  &  & 27.69 & 47.72 & 26.71 & 51.70 \\
\noalign{\smallskip}\hline
\end{tabular}
}
\end{table*}

For the sake of completeness, we present in Figure~\ref{fig:kernels} (denoising) and Figure~\ref{fig:kernelsDeblur} (deblurring) the resulting triplets of kernels (top), potential functions (middle) and activation functions (bottom) for a depth~$S=20$.
The scaling of the axes is identical among all potential functions and activation functions, respectively.
Note that the potential functions are computed by numerical integration of the learned activation functions and we choose the integration constant such that every potential function is bounded from below by 0.
As a result, we observe a large variety of different kernel structures, including bipolar forward operators in different orientations (e.g.~5$^{th}$~kernel in first row, 8$^{th}$~kernel in third row)
or pattern kernels representing prototypic image textures (e.g.~kernels in first column).
Likewise, the learned potential functions can be assigned to several representative classes of common regularization functions like, for instance, truncated total variation (8$^{th}$~function in second row of Figure~\ref{fig:kernels}),
truncated concave (4$^{th}$~function in third row of Figure~\ref{fig:kernels}), double-well potential (10$^{th}$~function in first row of Figure~\ref{fig:kernels}) or "negative Mexican hat" (8$^{th}$~function in third row of Figure~\ref{fig:kernels}).
Note that the associated kernels in both tasks nearly coincide, whereas the potential and activation functions significantly differ.
We observe that the activation functions in the case of denoising have a tendency to generate higher amplitudes compared to deblurring, which results in a higher relative balancing of the regularizer in the case of denoising.
\begin{figure*}
\includegraphics[width=\linewidth]{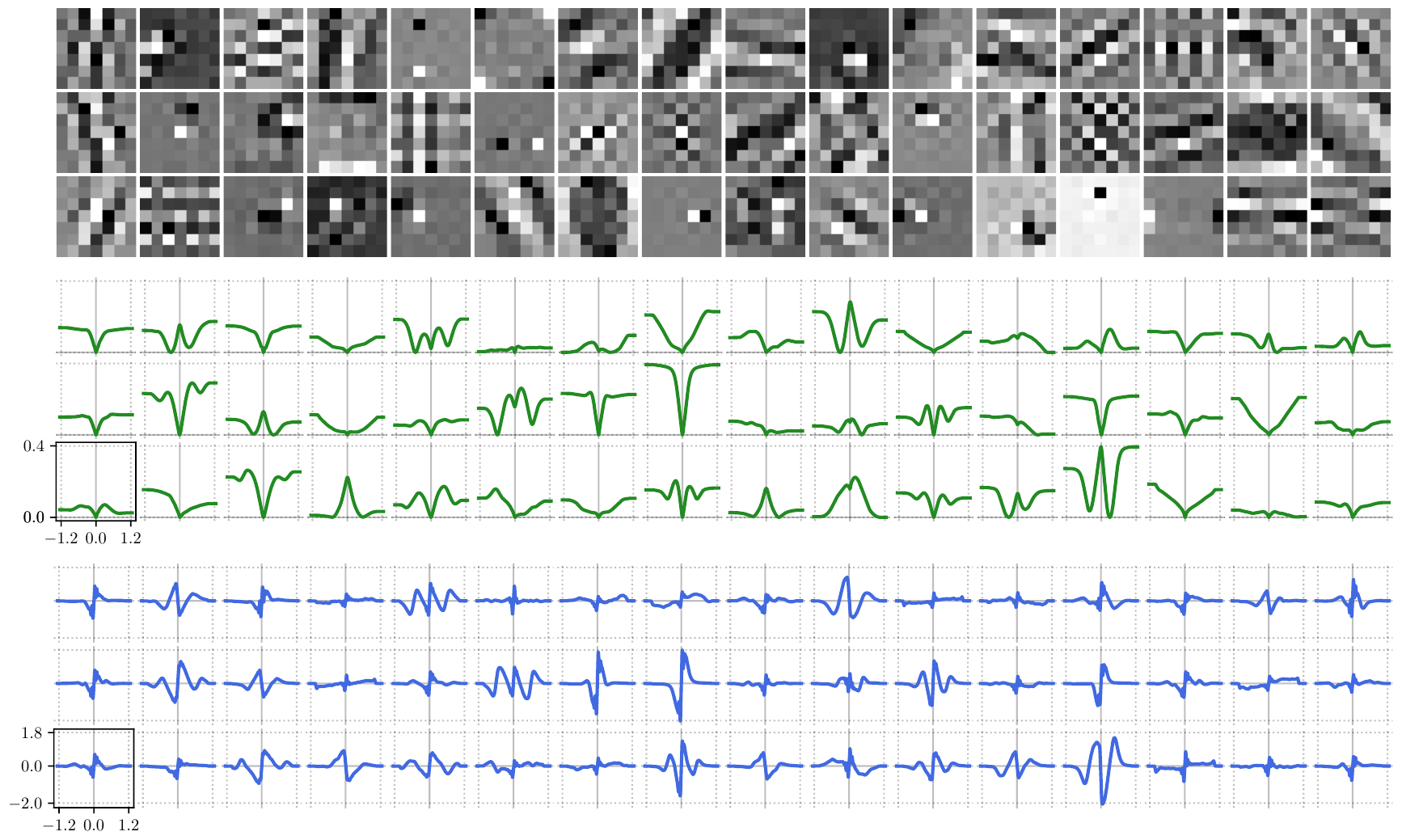}
\caption{Triplets of $7\times 7$-kernels (top), potential functions~$\rho$ (middle) and activation functions~$\phi$ (bottom) learned for image denoising.}
\label{fig:kernels}
\end{figure*}

\begin{figure*}
\includegraphics[width=\linewidth]{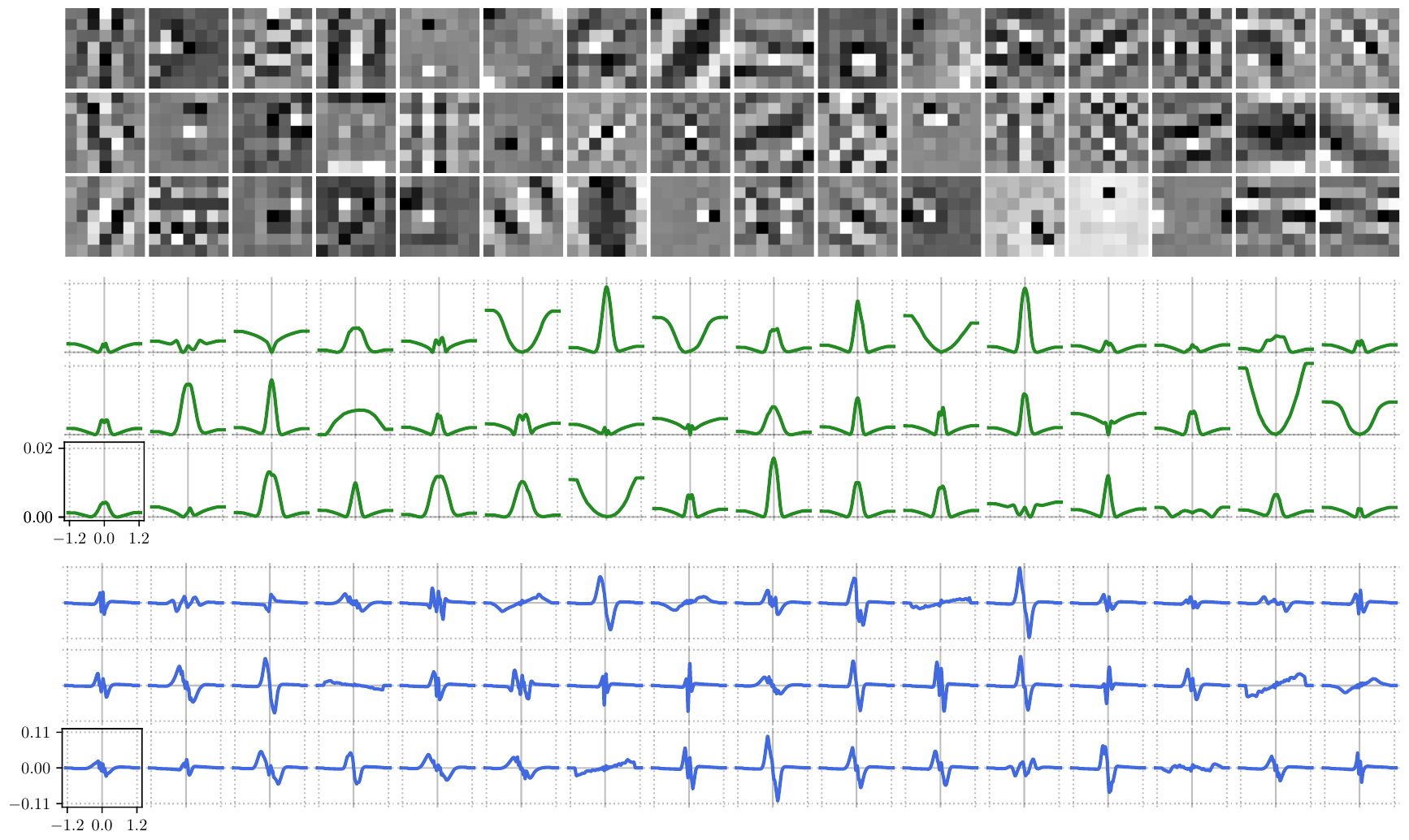}
\caption{Triplets of $7\times 7$-kernels (top), potential functions~$\rho$ (middle) and activation functions~$\phi$ (bottom) learned for image deblurring.}
\label{fig:kernelsDeblur}
\end{figure*}

\subsection{Spectral analysis of the learned regularizers}
Finally, in order to gain intuition of the learned regularizer, we perform a nonlinear eigenvalue analysis~\cite{Gi18} for the gradient of the Field of Experts regularizer learned for $S=20$ and $T=\overline{T}$.
For this reason, we compute several generalized eigenpairs $(\lambda_j,v_j)\in\R\times\R^{n}$ satisfying
\[
\sum_{k=1}^{N_K}K_k^\top\Phi_k(K_k v_j)=\lambda_j v_j
\]
for $j=1,\ldots,N_v$.
Note that by omitting the data term, the forward Euler scheme~\eqref{eq:Euler} applied to the generalized eigenfunctions~$v_j$ reduces to
\begin{equation}
v_j-\frac{T}{S}\sum_{k=1}^{N_K}K_k^\top\Phi_k(K_k v_j)=\left(1-\frac{\lambda_j T}{S}\right)v_j,
\label{eq:updateEigenfunction}
\end{equation}
where the contrast factor $(1-\frac{\lambda_j T}{S})$ determines the global intensity change of the eigenfunction.
We point out that due to the nonlinearity of the eigenvalue problem such a formula only holds locally for each iteration of the scheme.

We compute $N_v=64$~generalized eigenpairs of size $127\times 127$ by solving
\begin{align}
\min_{\{v_j\}_{j=1}^{N_v}}\sum_{j=1}^{N_v}\left\|\sum_{k=1}^{N_K}K_k^\top\Phi_k(K_k v_j)-\Lambda(v_j)v_j\right\|_2^2,
\label{eq:generalEVProblemLoss}
\end{align}
where
\[
\Lambda(v)=\frac{\left\langle\sum_{k=1}^{N_K}K_k^\top\Phi_k(K_k v),v\right\rangle}{\|v\|_2^2}
\]
denotes the generalized Rayleigh quotient, which is derived by minimizing \eqref{eq:generalEVProblemLoss} with respect to $\Lambda(v)$.
The eigenfunctions are computed using an accelerated gradient descent with step size control~\cite{PoSa16}.
All eigenfunctions are initialized with randomly chosen image pat\-ches of the test image data set, from which we subtract the mean.
Moreover, in order to minimize the influence of the image boundary, we scale the image intensity values with a Gaussian kernel.
We run the algorithm for $10^4$~iterations, which is sufficient for reaching a residual of approximately~$10^{-5}$ for each eigenpair.

Figure~\ref{fig:eigenpairsDenoising} depicts the resulting pairs of eigenfunctions and eigenvalues for image denoising.
We observe that eigenfunctions corresponding to smaller eigenvalues represent in general more complex and smoother image structures.
In particular, the first eigenfunctions can be interpreted as cartoon-like image structures with clearly separable interfaces.
Most of the eigenfunctions associated with larger eigenvalues exhibit texture-like patterns with a progressive frequency.
Finally, wave and noise structures are present in the eigenfunctions with the highest eigenvalues.

We remark that all eigenvalues are in the interval~$[0.025,11.696]$.
Since $\frac{T}{S}\approx 0.054$, the contrast factors $(1-\frac{\lambda_j T}{S})$ in \eqref{eq:updateEigenfunction}
are in the interval~$[0.368,0.999]$, which shows that the regularizer has a tendency to decrease the contrast.
Formula~\eqref{eq:updateEigenfunction} also reveals that eigenfunctions corresponding to contrast factors close to~$1$ are preserved over several iterations.
In summary, the learned regularizer has a tendency to reduce the contrast of high-frequency noise patterns, but preserves the contrast of texture- and structure-like patterns.

Figure~\ref{fig:eigenpairsDeblurring} shows the eigenpairs for the deblurring task.
All eigenvalues are relatively small and distributed around~$0$, which means that the corresponding contrast factors lie in the interval~$[0.992,1.030]$.
Therefore, the learned regularizer can both decrease and increase the contrast.
Moreover, most eigenfunctions are composed of smooth structures with a distinct overshooting behavior in the proximity of image boundaries.
This implies that the learned regularizer has a tendency to perform image sharpening.

\begin{figure*}
\includegraphics[width=\linewidth]{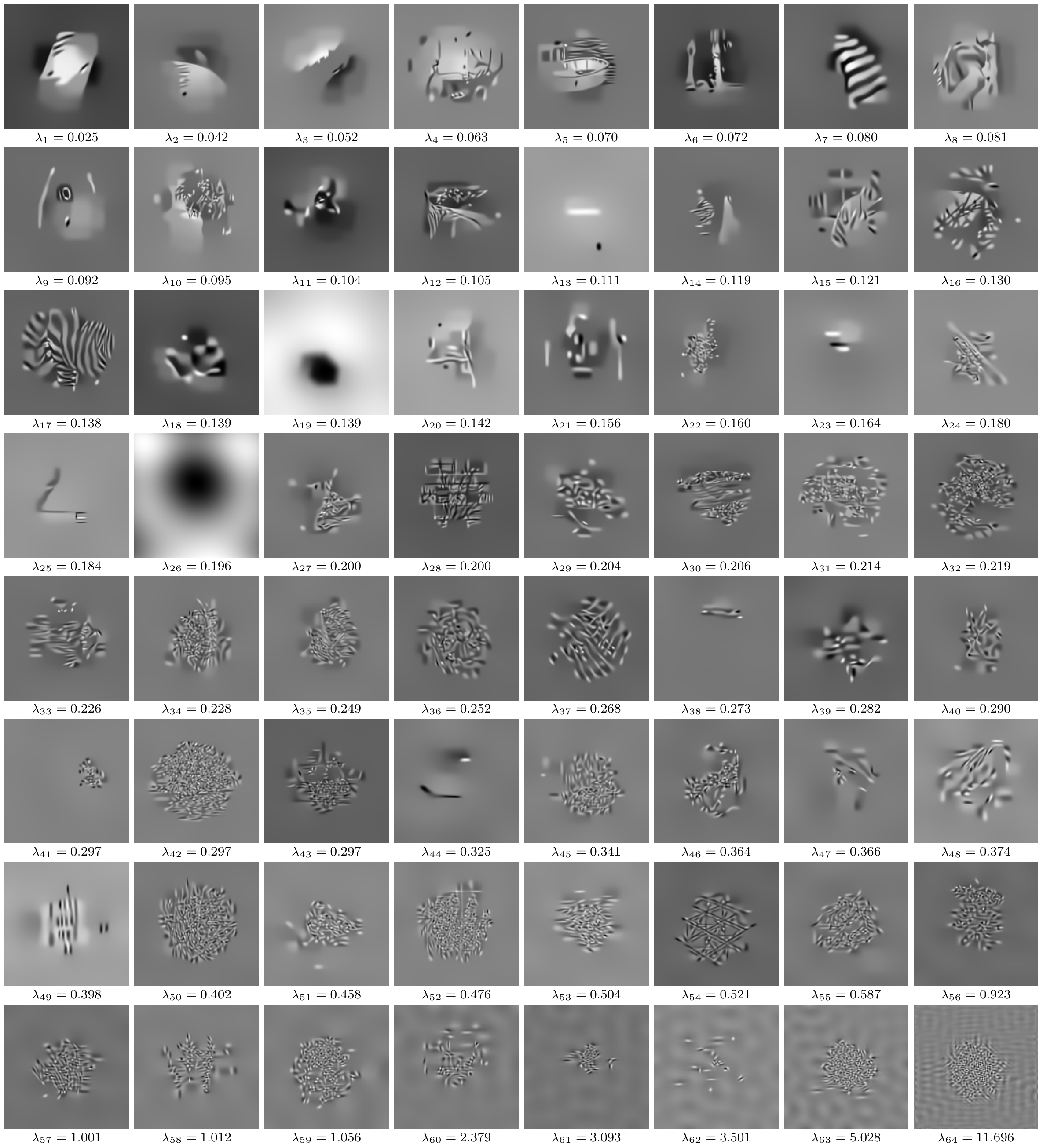}
\caption{$N_v=64$~eigenpairs for image denoising, where all eigenfunctions have the resolution $127\times 127$ and the intensity of each eigenfunction is adjusted to $[0,1]$.}
\label{fig:eigenpairsDenoising}
\end{figure*}

\begin{figure*}
\includegraphics[width=\linewidth]{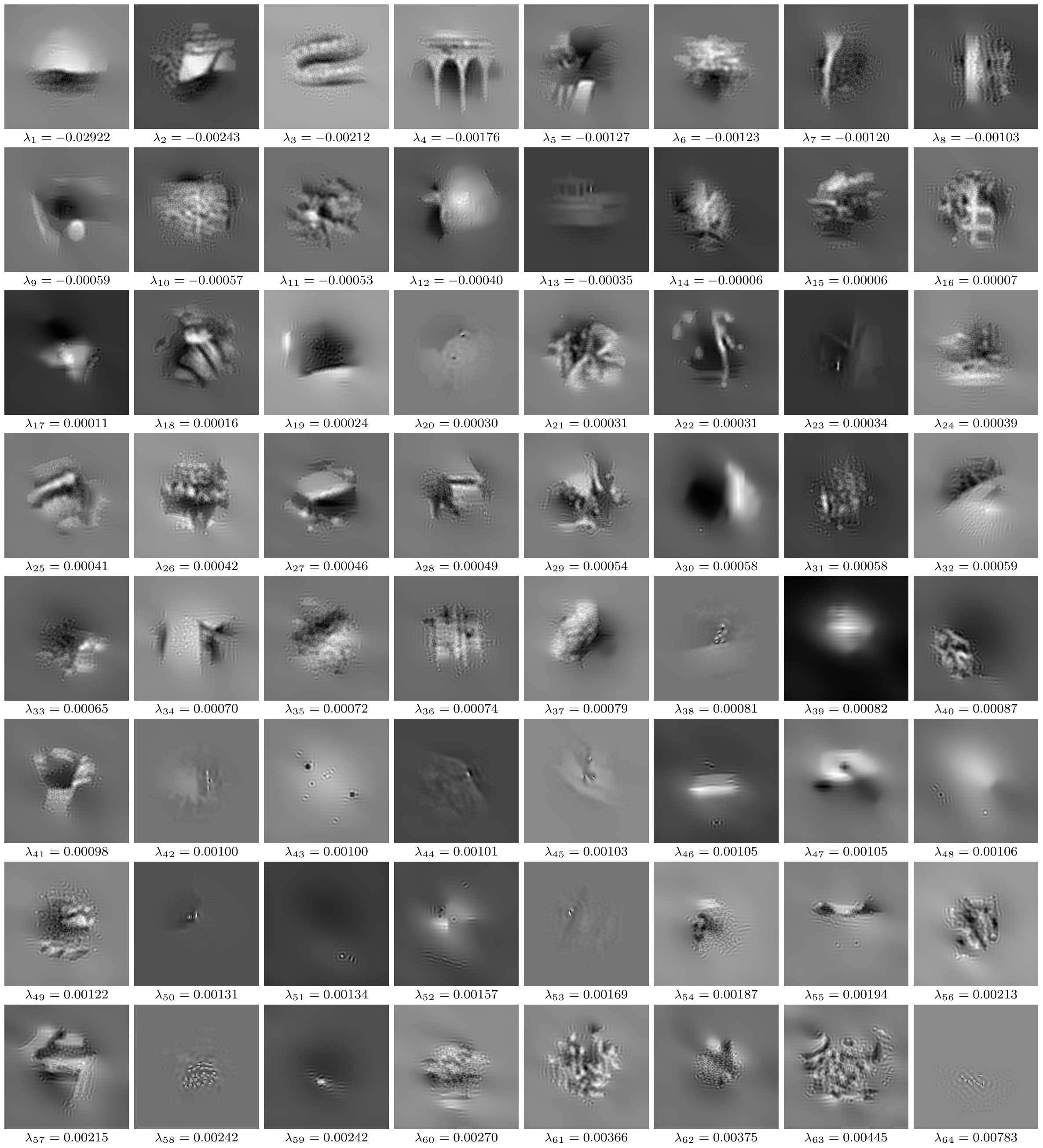}
\caption{$N_v=64$~eigenpairs for image deblurring, where all eigenfunctions have the resolution $127\times 127$ and the intensity of each eigenfunction is adjusted to $[0,1]$.}
\label{fig:eigenpairsDeblurring}
\end{figure*}

\section{Conclusion}
Starting from a parametric and autonomous gradient flow perspective of variational methods, we explicitly modeled the stopping time as a control variable in an optimal control problem.
By using a Lagrangian approach we derived a first order condition suited to automatize the calculation of the energy minimizing optimal stopping time.
A forward Euler discretization of the gradient flow led to static variational networks.
Numerical experiments confirmed that a proper choice of the stopping time is of vital importance for the image restoration tasks in terms of the PSNR value.
We performed a nonlinear eigenvalue analysis of the gradient of the learned Field of Experts regularizer, which revealed interesting properties of the local regularization behavior.
A comprehensive long-term spectral analysis in continuous time is left for future research.

\paragraph{Acknowledgements.}
We acknowledge support from the European Research Council under the Horizon 2020 program, ERC starting grant HOMOVIS (No. 640156) and ERC advanced grant OCLOC (No. 668998).

\end{document}